\documentclass[12pt,oneside]{amsart}

\usepackage{amssymb} 
\usepackage{enumerate, amsfonts, latexsym,epsfig, color, hyperref}
\usepackage{epstopdf}

\input xy
\xyoption{all}
\newcommand{\cd}[1]{\begin{equation*}{\xymatrix{#1}}\end{equation*}}
\newcommand{\cdlabel}[2]{\begin{equation}\label{#1}{\xymatrix{#2}}\end{equation}}

\newtheorem {theorem}{Theorem}[section]
\newtheorem {lemma} [theorem] {Lemma}
\newtheorem {proposition} [theorem] {Proposition}

\newtheorem {corollary} [theorem] {Corollary}

\theoremstyle{definition}
\newtheorem {definition} [theorem] {Definition}
\newtheorem {remark} [theorem] {Remark}
\newtheorem {example} [theorem] {Example}

\def\N {\mathbb N}

\def\ker {\mathrm{ker}}

\def\mc {\mathcal}

\def\Stab {\mathrm{Stab}}

\def\into {\, {\hookrightarrow}\, }

\newcommand{\bZ}{\mathbb{Z}}

\newcommand{\co}{\colon\thinspace}
\newcommand{\llangle}{\langle\negthinspace\langle}
\newcommand{\rrangle}{\rangle\negthinspace\rangle}

\usepackage{amsmath}

\newcommand{\leftQ}[2]{\left.\raisebox{-.2em}{$#2$}\middle\backslash\raisebox{.2em}{$#1$}\right.}

\makeatletter
\newtheorem*{rep@theorem}{\rep@title}
\newcommand{\newreptheorem}[2]{%
\newenvironment{rep#1}[1]{%
 \def\rep@title{#2 \ref{##1}}%
 \begin{rep@theorem}\it}%
 {\end{rep@theorem}}}
\makeatother

\makeatletter
\newtheorem*{rep@corollary}{\rep@title}
\newcommand{\newrepcorollary}[2]{%
\newenvironment{rep#1}[1]{%
 \def\rep@title{#2 \ref{##1}}%
 \begin{rep@corollary}\it}%
 {\end{rep@corollary}}}
\makeatother

\newreptheorem{theorem}{Theorem}
\newreptheorem{lemma}{Lemma}
\newreptheorem{corollary}{Corollary}

\newtheorem{thmA}{Theorem}

\newtheorem{corA}[thmA]{Corollary}

\begin{document}

\title[Quasiconvexity and Dehn filling]{Quasiconvexity and Dehn filling}

\author[D. Groves]{Daniel Groves}
\address{Department of Mathematics, Statistics, and Computer Science,
University of Illinois at Chicago,
322 Science and Engineering Offices (M/C 249),
851 S. Morgan St.,
Chicago, IL 60607-7045}
\email{groves@math.uic.edu}

\author[J.F. Manning]{Jason Fox Manning}
\address{Department of Mathematics, 310 Malott Hall, Cornell University, Ithaca, NY 14853}
\email{jfmanning@math.cornell.edu}

\thanks{Both authors thank the National Science Foundation (under grants DMS-1507067 and DMS-1462263) and the Simons Foundation (\#342049 to Daniel Groves and \#524176 to Jason Manning) for support.  Thanks also to the Mathematical Sciences Research Institute, where the second author was in residence during the conception of this work, and to the American Institute of Mathematics, where the paper was finished.
}

\begin{abstract}
We define a new condition on relatively hyperbolic Dehn filling which allows us to control the behavior of a relatively quasiconvex subgroups which need not be full.  As an application, in combination with recent work of Cooper and Futer \cite{CooperFuter}, we provide a new proof of the virtual fibering of non-compact finite-volume hyperbolic $3$--manifolds, a result first proved by Wise \cite{Wise}.  Additionally, we explain how the results of \cite[Appendix A]{VH} can be generalized to the relative setting to control the relative height of relatively quasiconvex subgroups under appropriate Dehn fillings.
\end{abstract}

\maketitle

\section{Introduction}  

Dehn filling results for hyperbolic and relatively hyperbolic groups have been used to great effect in recent years, notably in solving the isomorphism problem for a broad class of relatively hyperbolic groups \cite{dahmaniguirardel15}, and as part of Agol's proof of the Virtual Haken Conjecture \cite{VH} (see particularly the Appendix to \cite{VH} and \cite{AGM_msqt}).  In many of these results a key ingredient is the control of relatively quasiconvex subgroups under Dehn filling, building on techniques developed in \cite{agm}.

In previous work, this control was limited by the requirement that the fillings be `$H$--fillings', for a relatively quasiconvex subgroup $H$.  This requirement is mild when $H$ is \emph{full} (in the sense that each infinite  intersection of $H$ with a maximal parabolic subgroup is finite index in that parabolic), but more restrictive for general relatively quasiconvex subgroups.  One way to avoid this issue is to apply combination theorems such as those in \cite{MP09,MM-P}, etc. to enlarge relatively quasiconvex subgroups to full ones.  Even in case this is possible, the methods of this paper are conceptually simpler as they avoid this intermediate enlargement step.

In this paper, we propose a new condition on relatively hyperbolic Dehn filling, which we call {\em $H$--wide}, which is applicable to relatively quasiconvex subgroups $H$ in much greater generality than previous techniques.  We prove that under sufficiently long and $H$--wide fillings, the same control can be had over the behavior of a relatively quasiconvex subgroup $H$ under filling as could be obtained with sufficiently long $H$--fillings in the previous works.

We provide two main applications.  First, we combine our work with recent work of Cooper and Futer \cite{CooperFuter} to give a new proof of the following theorem of Wise.  

\begin{thmA}\label{t:VS} \cite[Theorem 14.29]{Wise}
 Suppose that $G$ is the fundamental group of a non-compact finite-volume hyperbolic $3$--manifold.  Then $G$ is virtually compact special.
\end{thmA}

The study of (virtually) special cube complexes and groups was initiated by Haglund and Wise in \cite{HW08}; we refer to that paper for the definition and basic properties.
The following is an immediate consequence of Theorem \ref{t:VS} and Agol's criterion for fibering \cite[Theorem 1.1]{agol:rfrs}.

\begin{corA} 
Suppose that $M$ is a non-compact finite-volume hyperbolic $3$--manifold.  Then $M$ has a finite-sheeted cover which fibers over the circle.
\end{corA}

  Previous to this paper, Wise's unpublished manuscript \cite{Wise} contained the only proof that a non-compact finite-volume hyperbolic $3$-manifold virtually fibers over a circle.  Our proof does not rely on any results from \cite{Wise} (and neither does the one in \cite{CooperFuter}).  
  
  The second application we provide is to use $H$--wide fillings to explain how the results from \cite[Appendix A]{VH} can be generalized to control the `relative height' of relatively quasiconvex subgroups under Dehn fillings.  We apply this to prove a result (Theorem \ref{t:Henry}) needed by Wilton and Zalesskii in their work \cite{WiltonZalesskii2} on profinite rigidity of $3$--manifold groups.

\subsection{On virtual fibering of hyperbolic $3$--manifolds}

Agol proved in \cite{VH} that fundamental groups of closed hyperbolic $3$--manifolds are virtually special, which implies that these manifolds are virtually Haken, and virtually fibered.  He also proved that Kleinian groups are LERF, and large.

In the non-compact but finite-volume case, the LERF and large results are included in Agol's result, and these manifolds are well known to be Haken.  However virtual fibering in the non-compact case is not covered by \cite{VH}.

We make some comments on Wise's proof of Theorem \ref{t:VS} in the most recent publicly available version of \cite{Wise}, dated October 29, 2012.
This proof relies on \cite[Theorem 16.28]{Wise}, which in turn uses \cite[Theorem 16.16]{Wise} in three places.  This last result is about the separability of quasiconvex subgroups of certain graphs of virtually sparse special groups.  In the proof of \cite[Theorem 16.16]{Wise}, Wise asserts that relative quasiconvexity of subgroups of relatively hyperbolic groups persists under sufficiently long fillings and refers to Osin \cite{osin:peripheral}.  Such a result is not in \cite{osin:peripheral} and in fact a correct formulation requires some care (see Example \ref{ex:long but not wide} below).

As explained above, existing results about controlling relatively quasiconvex subgroups under Dehn filling, such as those in \cite{agm, VH} and also \cite[Theorem 15.6]{Wise}, apply to relatively quasiconvex subgroups which are full, and they do not apply in the setting needed in \cite[Theorem 16.16]{Wise}.

We believe the above issue in the proof of \cite[Theorem 16.16]{Wise} can be fixed using either a `Combination Theorem' approach or techniques as in the current paper.  
However, one of our goals here is to use the advances of the last five years to give an alternative proof of virtual fibering in the non-compact setting.

\subsection{Outline}

In Section \ref{s:background} we recall the basic concepts about relatively hyperbolic groups, relatively quasiconvex subgroups, and Dehn filling.  In Section \ref{s:H-wide} we introduce the notion of $H$--wide fillings.  In Section \ref{s:Properties} we prove that the behavior of a relatively quasiconvex subgroup $H$ under sufficiently long and $H$--wide fillings is well-controlled.  We also give an example to show that this is not true without the assumption of $H$--wideness.  In Section \ref{s:H-wide exists} we prove that in certain circumstances we can ensure the existence of appropriate $H$--wide fillings.  In Section \ref{s:VF} we provide the application to virtual specialness of fundamental groups of finite-volume hyperbolic $3$--manifolds.  Finally, in Section \ref{s:Henry}, we prove that the results about height in the hyperbolic setting from \cite[Appendix A]{VH} can be generalized to control relative height under sufficiently long and $H$--wide fillings.  These results may be of independent interest.  As an application, we prove Theorem \ref{t:Henry}, the result required by Wilton and Zalesskii.

\subsection*{Acknowledgments}
Thanks to Stefan Friedl for asking for an alternative account of virtual fibering for finite-volume hyperbolic $3$--manifolds, and to Henry Wilton for useful discussions and for asking us to prove Theorem \ref{t:Henry}.  Thanks also to Dave Futer and Eduard Einstein for useful comments on an earlier version of this paper.

\section{Background} \label{s:background}

For background on relatively hyperbolic groups, their associated cusped spaces, and relatively hyperbolic Dehn filling see \cite{rhds}.  For background on relatively quasiconvex subgroups see \cite{agm,HruskaQC}.  We always work in a \emph{combinatorial cusped space} $X = X(G,\mc{P},S)$, where $S$ is some chosen generating set for $G$ which also contains generating sets for the peripheral groups $P\in \mc{P}$.  This cusped space contains a copy of the Cayley graph of $G$ with respect to the generators $S$.  The \emph{depth} of a vertex of $X$ is its distance to the Cayley graph.

Suppose that $(G,\mc{P})$ is relatively hyperbolic.  We fix a combinatorial cusped space $X$ for the pair $(G,\mc{P})$ as in \cite{rhds}.  Since $(G,\mc{P})$ is relatively hyperbolic, $X$ is Gromov hyperbolic.  Unless otherwise stated, $\delta$ is a hyperbolicity constant for the space $X$.

Recall that a {\em Dehn filling} of $(G,\mc{P})$ is determined by a collection $\mc{N} = \{ N_P \}_{P \in \mc{P}}$ of normal subgroups $N_P \lhd P$.  The Dehn filling is the quotient
\[	G(\mc{N}) = G / \llangle \cup N_P \rrangle	.	\]
The {\em peripheral groups} of $G(\mc{N})$ are the images of the elements of $\mc{P}$ in $G(\mc{N})$.  We often abbreviate this as $(G,\mc{P}) \to (\overline{G},\overline{\mc{P}})$.  A statement $\mathsf{S}$ holds {\em for all sufficiently long fillings}
if there is a finite set $B \subset \cup P \smallsetminus \{ 1 \}$ so that $\mathsf{S}$ holds for any fillings $G(\mc{N})$ so that $N_P \cap B = \emptyset$ for all $P \in \mc{P}$. 
If $(\overline{G},\overline{\mc{P}})$ is a Dehn filling of $(G,\mc{P})$, with $\overline{G} = G/K$, then one obtains a combinatorial cusped space $\overline{X}$ for $(\overline{G},\overline{\mc{P}})$ by taking $\overline{X}$ equal to $\leftQ{X}{K}$ with self-loops removed.  In fact, since the self-loops do not affect the metric on the zero-skeleton, we ignore the issue of removing them and abuse notation by setting $\overline{X} = \leftQ{X}{K}$.  

 The following result is key to any approach to relatively hyperbolic Dehn filling theorems using the cusped space.

\begin{theorem} \label{t:large balls embed}
Using the cusped spaces just described, let $B$ be a finite metric ball in $X$.  For all sufficiently long fillings the quotient map $X \to \leftQ{X}{K}$ restricts on $B$ to an isometric embedding whose image is a metric ball.
\end{theorem}
\begin{proof}
This follows immediately from \cite[Theorem 1.1]{osin:peripheral}.
\end{proof}

Using the coarse Cartan-Hadamard Theorem \cite[A.1]{Coulon} and the uniform hyperbolicity of combinatorial horoballs \cite[Theorem 3.8]{rhds} we obtain the following corollary, which was stated in a slightly weaker form as \cite[Proposition 2.3]{agm}.

\begin{proposition}\label{p:agm 2.3}
  For all $\delta>0$ there is a $\delta'>0$ so that if the combinatorial cusped space $X$ is $\delta$--hyperbolic, then for all sufficiently long fillings, the combinatorial cusped space $\overline{X}$ of the Dehn filling is $\delta'$--hyperbolic.
\end{proposition}

\begin{remark}
Proposition \ref{p:agm 2.3} implies that $(\overline{G},\overline{\mc{P}})$ is relatively hyperbolic, using only Theorem \ref{t:large balls embed}.  The proof of Theorem \ref{t:large balls embed} in \cite{rhds} used the bicombing of $X$ by {\em preferred paths}, whereas the proof that the cusped space of $\overline{X}$ is Gromov hyperbolic used a homological bicombing which used an adaptation of results of Mineyev from \cite{min:str}.  Using the above approach allows one to avoid the homological bicombing in \cite{rhds} entirely.
\end{remark}

\begin{definition}
  Suppose that $(G,\mc{P})$ is relatively hyperbolic with associated cusped space $X$.  Let $A$ be a horoball in $X$, and let $R>0$.  A geodesic \emph{penetrates $A$ to depth $R$} or \emph{$R$--penetrates $A$} if it contains a point in $A$ at depth $R$.

  Suppose $H \le G$.  Then $A$ is {\em $R$--penetrated by $H$} if there is a geodesic $\gamma$ with endpoints in $H$ which $R$--penetrates $A$.
\end{definition}

Recall the following result from \cite{MM-P}.

\begin{proposition}\label{prop:MMP A.6} \cite[Proposition A.6]{MM-P}
Let $(G,\mc{P})$ be relatively hyperbolic and $H \le G$ be relatively quasiconvex.  There is a constant $R$ so that whenever a horoball $A$ of $X$ is $R$--penetrated by $H$ then the intersection of $H$ with the stabilizer of the horoball is infinite.
\end{proposition}

  The following is a combination of Lemma 3.3 from \cite{GM-splittings}, and a statement implicit in its proof.
\begin{lemma}\label{l:GM-splittings 3.3} \cite[Lemma 3.3]{GM-splittings}
  Suppose $(G,\mc{P})$ is relatively hyperbolic, with $\delta$--hyperbolic combinatorial cusped space $X$.  Suppose further that $P_1$ and $P_2$ are distinct conjugates of elements of $\mc{P}$, and that $F = P_1 \cap P_2$.  Then $F$ acts freely on some set $Q$ in $X$ which lies in the Cayley graph and has diameter (in $X$) at most $2\delta + 1$.

  In particular, there is a constant $C$ depending only on $\delta$ and the cardinality of the generating set $S$ so that $\# F\le C$.
\end{lemma}

The second part of the Lemma says that if $P_1$ and $P_2$ are distinct maximal parabolics, then $\#P_1\cap P_2 \le C$.
In other words, the family $\mc{P}$ is {\em $C$--almost malnormal}.   In particular, for a parabolic subgroup $A$ of size more than $C$, there is no ambiguity about which $g \in G$ and which $P \in \mc{P}$ has $A \le P^g$ (up to the choice of conjugating element in $gP$).

The following result was stated without proof as \cite[Proposition 3.4]{GM-splittings}.  The proof that we provide here is more elementary than the one suggested in \cite{GM-splittings}.

\begin{proposition}\label{p:GM-splittings 3.4}
 If $(G,\mc{P})$ is relatively hyperbolic, and $\mc{P}$ is $C$--almost malnormal, then for all sufficiently long fillings $(\overline{G},\overline{\mc{P}})$ of $(G,\mc{P})$, the collection $\overline{\mc{P}}$ is $C$--almost malnormal.
\end{proposition}
\begin{proof}
 By Proposition \ref{p:agm 2.3}, there is a $\delta'$ so that for all sufficiently long fillings $(\overline{G},\overline{\mc{P}})$ of $(G,\mc{P})$, the cusped space $\overline{X}$ for $(\overline{G},\overline{\mc{P}})$ is $\delta'$--hyperbolic.  Fix a filling $(\overline{G},\overline{\mc{P}})$ so that the induced map between cusped spaces is injective on any ball of radius $100\delta'$ centered in the Cayley graph of $X$ (see Theorem \ref{t:large balls embed}).
 
 Suppose that $\overline{P}_1$ and $\overline{P}_2$ are distinct conjugates of elements of $\overline{\mc{P}}$, and let $\overline{F} = \overline{P}_1 \cap \overline{P}_2$.  There are horoballs $\overline{A}_1$ and $\overline{A}_2$ in $\overline{X}$ so that $\overline{P}_i$ stabilizes $\overline{A}_i$.  As explained in the proof of \cite[Lemma 3.3]{GM-splittings}, the subgroup $\overline{F}$ acts freely on a subset of the Cayley graph of $\overline{G}$ in $\overline{X}$ of diameter at most $2\delta'+1$.  Moreover, it is clear by considering the $\overline{F}$--orbit of a geodesic between the limit points of $\overline{A}_1$ and $\overline{A}_2$ in $\partial X$ that there are also subsets $Q_1$ and $Q_2$ of diameter at most $2\delta'+1$ so that $\overline{F}$ acts freely on each $Q_i$ and $Q_i$ is contained in $\overline{A}_i$ at depth $5\delta'$.
 
 Suppose that a geodesic between $Q_1$ and $Q_2$ $10\delta'$--penetrates some other horoball $B$.  Then let $B$ be the closest such horoball to $\overline{A}_1$, and replace $\overline{A}_2$ by $B$ and $Q_2$ by an $\overline{F}$--invariant subset $Q$ of $B$ at depth $5\delta'$ and diameter at most $2\delta'+1$.  In this manner, we may suppose that any geodesic between $Q_1$ and $Q_2$ stays within a $10\delta'$--neighborhood of the Cayley graph in $\overline{X'}$.
 
 We may thus lift $Q_1$, $Q_2$ and the geodesics between them to $X$.  To see that this is possible, consider that any pair of points in $Q_1$ and pair of points in $Q_2$ are the vertices of a geodesic quadrilateral with two sides of length at most $5\delta'$ and so can be filled with a disk which lies entirely within a $20\delta'$--neighborhood of the Cayley graph.  In an entirely similar way to the proof of \cite[Theorem 4.1]{GM-splittings} (a result whose proof did not rely on the result we are currently trying to prove), it now follows that $\overline{F}$ can be lifted bijectively to a finite subgroup $F$ of $G$ which stabilizes two distinct horoballs.  Because $\mc{P}$ is $C$--almost malnormal, it follows that $|\overline{F}| = |F| \le C$, which is what we were required to prove.
\end{proof}

\begin{definition}
 Suppose that $(G,\mc{P})$ is relatively hyperbolic and that $X$ is a cusped space for $(G,\mc{P})$ which is $\delta$--hyperbolic.  A parabolic subgroup $Q$ of $G$ is {\em uniquely parabolic} if there is a unique conjugate of an element of $\mc{P}$ which contains $Q$.
\end{definition}

It follows from \cite[Lemma 3.3]{GM-splittings} that there is a constant $C$ so that any parabolic subgroup of size more than $C$, and in particular any infinite parabolic subgroup, is uniquely parabolic.

It is an immediate consequence of the definition that a uniquely parabolic subgroup stabilizes a unique horoball in the cusped space.

In order to fix notation, we recall a definition from \cite[Section 3]{agm}, and slightly adapt the notation from there.
Let $H \le G$.  Suppose that $\mc{D}$ is a collection of representatives of $H$--conjugacy classes of maximal uniquely parabolic subgroups of $H$.  Given $D \in \mc{D}$, there exists $P_{D} \in \mc{P}$ and $c_D \in G$ so that $D \le c_DP_{D}c_D^{-1}$.  We fix such $c_D$, and suppose that $c_D$ is a shortest possible choice.  We abuse notation slightly and write $(H,\mc{D}) \le (G,\mc{P})$. Let $Y$ be a combinatorial cusped space for the pair $(H,\mc{D})$.  
The inclusion $\iota \co H \into G$ extends to an $H$--equivariant Lipschitz map $\check{\iota} \co Y^{(0)} \to X$ as follows:

A vertex in a horoball of $Y$ is determined by a triple $(sD,h,n)$ where $s \in H$, $D \in \mc{D}$ and $n \in \mathbb N$.  We define
\[	\check{\iota} (sD,h,n) = (sc_DP_{D},hc_D,n)	.	\]
It follows from \cite[Lemma 3.1]{agm} that $\check{\iota}$ is $H$--equivariant and $\alpha$--Lipschitz for some $\alpha$.  We refer to $\check{\iota}$ as the {\em induced map on cusped spaces}.  Whenever we have a pair $(H,\mc{D}) \le (G,\mc{P})$ as above, we fix the subgroups $P_D \in \mc{P}$ and the (shortest) elements $c_D$ as above.

\begin{definition}
 Suppose that $(G,\mc{P})$ is relatively hyperbolic and that $H \le G$ is a subgroup.  Suppose that $\mc{D}$ consists of a set of representatives of $H$--conjugacy classes of maximal uniquely parabolic subgroups of $H$.  Then $(H,\mc{D})$ is {\em relatively quasiconvex} in $(G,\mc{P})$ if the image of the $0$--skeleton of the cusped space of $(H,\mc{D})$ in the cusped space of $(G,\mc{P})$ is $\lambda$--quasiconvex for some $\lambda$.  In this case we say that $\lambda$ is a \emph{quasiconvexity constant for $(H,\mc{D})$ in $(G,\mc{P})$}.
\end{definition}
This definition is slightly different than the one in \cite{agm}, since we do not assume that $(H,\mc{D})$ is relatively hyperbolic.  However, we do assume that $\mc{D}$ consists of maximal uniquely parabolic subgroups of $H$.  If the image of the cusped space of $(H,\mc{D})$ in the cusped space of $(G,\mc{P})$ is quasiconvex, then it follows from the proof of \cite[Theorem A.10]{MM-P} that $H$ is relatively quasiconvex in the sense of Hruska \cite{HruskaQC}, and hence that $(H,\mc{D})$ is relatively hyperbolic.  Therefore, this definition is equivalent to others in the literature, by the results in \cite[Appendix A]{MM-P}.

\section{$H$-wide fillings}\label{s:H-wide}

\begin{definition}
Let $P$ be a group, $B \le P$ a subgroup, and $S$ a finite set.  A normal subgroup $N \unlhd P$ is {\em $(B,S)$--wide in $P$} if whenever there are $b \in B$ and $s \in S$ so that $bs \in N$ we have $s \in B$.
\end{definition}

\begin{definition} \label{d:wide}
Let $(G,\mc{P})$ be relatively hyperbolic and let $(H,\mc{D}) \le (G,\mc{P})$ be relatively quasiconvex.  Let $S \subseteq \bigcup \mc{P} \smallsetminus \{ 1 \}$.  A filling 
\[	G \to \overline{G} = G(\mc{N})	\]
is {\em $(H,S)$--wide} if for any $D \in \mc{D}$ (with $D \le P_{D}^{c_D}$ as above) the normal subgroup $N_{D}$ is $\left(D^{c_D^{-1}}, S \cap P_D\right)$--wide in $P_D$. (To simplify notation, for $D \in \mc{D}$, we write $N_D$ for $N_{P_D}$.)
\end{definition}

Since it is possible that $P_{D_1} = P_{D_2}$ for $D_1 \ne D_2$, it is also possible that $N_{D_1} = N_{D_2}$.  We also remark that $N_{D}$ need not be a subgroup of $H$.

In place of the statement in Definition \ref{d:wide} above, we sometimes use the equivalent formulation that for any $D \in \mc{D}$ (with $D \le P_D^{c_D}$ as above), any $d \in D$ and any $w \in S \cap P_{D}$, if $d c_Dw c_D^{-1} \in N_D^{c_D}$, then $c_Dw c_D^{-1} \in  D$.

\begin{definition}
We say that a property $\mathsf{P}$ holds {\em for all sufficiently long and $H$-wide fillings} if there is a finite set $S \subseteq \bigcup \mc{P} \smallsetminus \{ 1 \}$ so that $\mathsf{P}$ holds for any $(H,S)$--wide filling $G \to G(\mc{N})$ for which $N \cap S = \emptyset$ for each $N \in \mc{N}$.
\end{definition}

\begin{remark}
  In the definition of $(H,S)$--wide, one should think of $S$ containing all nontrivial elements of $\bigcup \mc{P}$ in a large ball around the identity.  This ensures that, for each $D\in \mc{D}$, a ``big neighborhood'' of $\overline{D}=D/(N_{D}^{c_D}\cap D)$ embeds in $\overline{P}_{D}^{c_D} = (P_{D}/N_{D})^{c_D}$, ruling out behavior like that pictured in Figure \ref{fig:Hwide}.
\end{remark}
\begin{figure}[htbp]
  \centering
  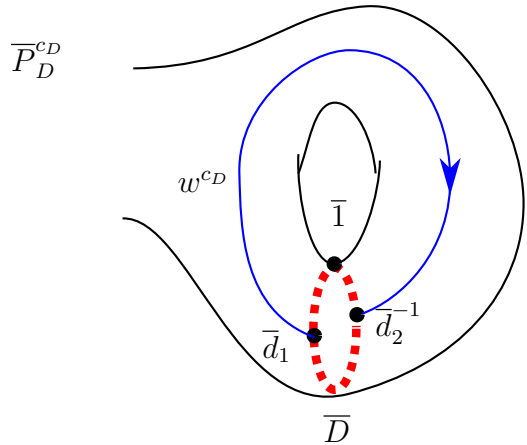
  \caption{A cartoon of the coset graph for $N_{D}^{c_D}$ in $P_{D}^{c_D}$
   and the kind of loop forbidden by $(H,S)$--wideness with $w\in S$.}
  \label{fig:Hwide}
\end{figure}

Previous quasiconvex Dehn filling results \cite{VH, agm, MM-P} have been in terms of ``$H$--fillings'', whose definition we now recall.
\begin{definition}\label{def:Hfill}
  Let $(G,\mc{P})$ be relatively hyperbolic, let $H<G$ be relatively quasiconvex, and let $\mc{N}=\{N_P\}_{P\in\mc{P}}$ be a collection of filling kernels.  The Dehn filling $G\to \overline{G}=G(\mc{N})$ is said to be an \emph{$H$--filling} if, whenever $\#(P^g\cap H)=\infty$, the kernel $N_P^g$ lies entirely in $H$.
\end{definition}
  
\begin{remark}
In \cite{agm} the condition `$P^g\cap H\neq\{1\}$' was used instead of `$\#(P^g\cap H)=\infty$.' As explained in \cite{MM-P}, the formulation in Definition \ref{def:Hfill} is the correct one if there is torsion, and this is the definition that is used in \cite{VH,MM-P}.
\end{remark}

The following result shows that, at least for long fillings, the notion of $H$--wide filling generalizes that of $H$--filling.

\begin{lemma} \label{l:long H-fill is wide}
  Let $(G,\mc{P})$ be relatively hyperbolic, and let $H<G$ be relatively quasiconvex.
  For any finite $S \subset G$ any sufficiently long $H$--filling is $(H,S)$--wide.
\end{lemma}
\begin{proof}
  Let $R$ be the constant from Proposition \ref{prop:MMP A.6}, as applied to $H$.  Let $\mc{D}$ be the peripheral structure on $H$ consisting of maximal uniquely parabolic subgroups, and $\{P_{D}\in \mc{P}\}$ and $\{c_D\in G\}$ the elements described before, so that $D<P_{D}^{c_D}$ for each $D\in \mc{D}$.   Let $M = \max\{d_X(1,c_D)\}+ \max\{d_X(1,w)\mid w\in S\}$.
Choose filling kernels $\{N_j\lhd P_j\}$ determining a sufficiently long $H$--filling so that any geodesic joining $1$ to $n\in N_j\setminus \{1\}$ must $(R+M+2\delta+2)$--penetrate the horoball stabilized by $P_j$.

Now suppose that for some $w\in S$ and some $d\in D\in \mc{D}$ we have $dw^{c_D}\in N_D^{c_D}$.  We must show that $w^{c_D}\in D$.  Since $(dw^{c_D})^{c_D^{-1}}\in N_{D}$, the geodesic from $1$ to $(dw^{c_D})^{c_D^{-1}}$ must $(R+M+2\delta+2)$--penetrate the horoball stabilized by $P_{D}$.  In particular, $d_X(1,(dw^{c_D})^{c_D^{-1}})$ must be at least $2R + 2M + 4\delta + 4$.  Consider the quadrilateral with vertices $1,c_D,dc_D w, d$.  The segment $[c_D,d c_D w]$ is the translate of a geodesic $[1,(dw^{c_D})^{c_D^{-1}}]$ by $c_D$, so it has length at least $2R + 2M + 4\delta + 4$ and $(R+M+2\delta+2)$--penetrates the horoball based on $c_DP_{D}$.  Since the sides $[1,c_D]$ and $[dc_D w, d]$ have length at most $M$, the side $[1,d]$ must pass within $2\delta$ of $[c_D,d c_D w]$ at its midpoint, which is also its deepest point in the horoball.  In particular $[1,d]$ must $R$--penetrate the horoball based on $c_D P_{D}$.  Since $d\in H$, Proposition \ref{prop:MMP A.6} implies that $H\cap P_D^{c_D} = D$ is infinite.  Since the filling kernels $\mc{N}$ determine an $H$--filling, this implies that $N_D<D$, and in particular, the element $dw^{c_D}\in D$.  It immediately follows that $w^{c_D}\in D$, so we have established that the filling is $(H,S)$--wide.
\end{proof}

In any case, if $(H,\mc{D})$ is a relatively quasiconvex subgroup of the relatively hyperbolic pair $(G,\mc{P})$, any Dehn filling of $(G,\mc{P})$ induces a Dehn filling of $(H,\mc{D})$, which may or may not inject into the filling of $G$.
\begin{definition}\label{def:inducedfilling}
  Let $(G,\mc{P})$ be relatively hyperbolic, and let $H<G$ be relatively quasiconvex.  Let $\mc{D}$ be the canonical (uniquely parabolic) peripheral structure on $H$, so each $D\in \mc{D}$ is contained in some $P_D^{c_D}$ for a unique $P_D\in \mc{P}$, and some shortest $c_D$.  Let $\mc{N} = \{N_P\}_{P\in \mc{P}}$ be a collection of filling kernels for $(G,\mc{P})$.  The \emph{induced filling kernels} for $(H,\mc{D})$ are the collection $\mc{N}_H = \{N_{P_D}^{c_D}\cap D\}_{D\in \mc{D}}$.  These define the \emph{induced filling}
\[ (H,\mc{D})\longrightarrow (H(\mc{N}_H),\overline{\mc{D}}), \]
where $\overline{\mc{D}}$ consists of the images of the elements of $\mc{D}$ in $H(\mc{N}_H)= H/\llangle \bigcup \mc{N}_H\rrangle_H$.
\end{definition}

There is a natural map from $H(\mc{N}_H)$ to the filling $G(\mc{N})$.
\section{Properties of $H$--wide fillings} \label{s:Properties}

In this section we prove various results which imply that a relatively quasiconvex subgroup $H$ can be controlled in $H$--wide fillings.  These results should be compared to those in \cite[Section 4]{agm}, where analogous results are proved for the behavior of a full relatively quasiconvex subgroup $H$ under sufficiently long $H$--fillings.

Let $(G,\mc{P})$ be relatively hyperbolic.  According to Proposition \ref{p:agm 2.3}, there exists a constant $\delta$ so that the cusped space for $G$ is $\delta$--hyperbolic, and moreover the induced cusped spaces for sufficiently long fillings of $(G,\mc{P})$ are also $\delta$--hyperbolic.  In this section, we assume that $\delta$ is such a constant, and that all fillings we perform are long enough so that the cusped spaces of the filled groups are $\delta$--hyperbolic.

The following lemma is a reformulation of \cite[Lemma 4.1]{agm}.

\begin{lemma} \label{l:4.1}
Suppose $(G,\mc{P})$ is relatively hyperbolic, and that $L_1, L_2  \ge 10\delta$.  For sufficiently long fillings $\pi \co G \to \overline{G} = G/K$ with induced map between cusped spaces $\pi \co X \to \overline{X}$, and any geodesic $\gamma$ in $X$ either:
\begin{enumerate}
\item There is a $10\delta$--local geodesic in $\overline{X}$ between the endpoints of $\pi(\gamma)$ which lies in a $2$--neighborhood of $\pi(\gamma)$ and agrees with $\pi(\gamma)$ in the $L_1$--neighborhood of the Cayley graph in $\overline{X}$; or
\item There is a horoball $A$ in $X$ so that $\gamma$ $L_2$--penetrates $A$ in a segment $[x,y]$ with $x,y\in G$, and there is some $k \in K$ stabilizing $A$ so that $d_X(x,k.y) \le 2L_1+3$.
\end{enumerate}
\end{lemma}

The following result is very similar to \cite[Lemma 4.2]{agm} but for $H$--wide fillings rather than $H$--fillings.  The \emph{induced filling} is defined above in Definition \ref{def:inducedfilling}.
\begin{lemma}\label{l:4.2}
Suppose that $(G,\mc{P})$ is relatively hyperbolic and that $H \le G$ is relatively quasiconvex.  Let $R$ be the constant from Proposition \ref{prop:MMP A.6}.  For any $L_1 \ge 10\delta$, $L_2 > \max \{ 2L_1+3, R \}$, and all sufficiently long and $H$--wide fillings $\pi \co G \to \overline{G}$, the following holds: suppose that $K_H \le \ker(\pi) \cap H$ is the kernel of the induced filling of $H$, that $h \in H$ and that $\gamma$ is a geodesic from $1$ to $h$.  If conclusion (2) of Lemma \ref{l:4.1} holds then there exists $k \in K_H$ so that $d_X(1,kh) < d_X(1,h)$.
\end{lemma}
\begin{proof} Let $\mc{D}$ be a collection of representatives of $H$--conjugacy classes of maximal uniquely parabolic subgroups of $H$, so that $(H,\mc{D})$ is relatively quasiconvex in $(G,\mc{P})$.

Let $\gamma$ be a geodesic as in the statement of the lemma, and suppose that conclusion (2) of Lemma \ref{l:4.1} holds.  Accordingly there is some horoball $A$ which is $L_2$--penetrated by $\gamma$.  Let $gP$ be the coset on which $A$ is based.
According to Proposition \ref{prop:MMP A.6}, $H \cap P^g$ is infinite.  This implies that there are $r \in H$, and $D \in \mc{D}$, so that $P = P_{D}$ and $gP = rc_DP_{D}$.  The intersection of $\gamma$ with $A$ is the segment $[x,y]$, and there is an element $k \in N_{D}^{rc_D}$ so that $d_X(x,k.y) \le 2L_1+3$.

Now, by quasiconvexity of $(H,\mc{D})$, there exists some $d_1,d_2 \in D$ so that $d_X(x,rd_1c_D), d_X(y,rd_2c_D)$ are both bounded by some constant $L$ depending only on the quasiconvexity constant for $(H,\mc{D})$.

Let $w_1 = (rd_1c_D)^{-1}ky$ and $w_2 = y^{-1}rd_2c_D$.  Both $d_X(1,w_1)$ and $d_X(1,w_2)$ are at most $L + 2L_1 + 3$, and both $w_1$ and $w_2$ lie in $P_D$.    Let $S$ be the set of words in the parabolic subgroups of $X$--length at most $2(L+2L_1+3)$.
Since $k \in N_{D}^{rc_D}$, we can find $n \in N_{D}$ so that $k = rc_Dnc_D^{-1}r^{-1}$.

We have
\begin{eqnarray*}
k &=& ky . y^{-1} \\
&=& rd_1c_D w_1w_2 c_D^{-1}d_2^{-1}r^{-1}.
\end{eqnarray*}
Therefore,
\[	c_Dnc_D^{-1} = d_1c_Dw_1w_2c_D^{-1}d_2^{-1}	,	\]
and 
\[ (d_2^{-1}c_D)n(d_2^{-1}c_D)^{-1} = (d_2^{-1}d_1)c_D w_1 w_2 c_D^{-1}.\]
However, for an $(H,S)$--wide filling, there can only be an element of $N_{D}$ of this form if $c_Dw_1w_2c_D^{-1} \in D$.  This implies that $c_Dnc_D^{-1} \in D$, which implies that $k \in D^r \cap N_{D}^{c_D} \le K_H$.

Since $d_X(x,ky) \le 2L_1+3$, but $d_X(x,y) \ge 2L_2 > 2L_1+3$, it is clear that $d_X(1,k.h) < d_X(1,h)$, as required.
\end{proof}

The following result is an immediate consequence.

\begin{corollary} \label{c:shortest}
Suppose that $(G,\mc{P})$ is relatively hyperbolic and that $H \le G$ is relatively quasiconvex.  For any $L \ge 10\delta$ and for all sufficiently long and $H$--wide fillings $\pi \co G \to \overline{G}$, if $h \in H$ is the shortest element of $H \cap \pi^{-1}(\pi(h))$ and $\gamma$ is a geodesic from $1$ to $h$ then there is a $10\delta$--local geodesic in $\overline{X}$ with the same endpoints as $\pi(\gamma)$ which lies in a $2$--neighborhood of $\pi(\gamma)$ and agrees with $\pi(\gamma)$ in an $L$--neighborhood of the Cayley graph of $\overline{G}$ in $\overline{X}$.
\end{corollary}

Recall that in a $\delta$--hyperbolic space, $10\delta$--local geodesics are quite close to geodesics.  In particular, we have the following (see \cite[III.H.1.13]{bridhaef:book} for a more general and precise statement):
\begin{lemma}\label{l:local geodesic}
  Let $\gamma$ be a $10\delta$--local geodesic in a $\delta$--hyperbolic space.  Then $\gamma$ is a $(7/3,2\delta)$--quasigeodesic, and is Hausdorff distance at most $3\delta$ from any geodesic with the same endpoints.
\end{lemma}

Lemma \ref{l:4.2}, and its interpretation in the form of Corollary \ref{c:shortest} are the key results needed to generalize many results about $H$--fillings to sufficiently long and $H$--wide fillings, as we now explain.

\begin{proposition} [cf. Proposition 4.3, \cite{agm}] \label{p:pi(H) QC}
Let $(G,\mc{P})$ be relatively hyperbolic and suppose that $H$ is a relatively quasiconvex subgroup of $(G,\mc{P})$, with relative quasiconvexity constant $\lambda$.  There exists $\lambda' = \lambda'(\lambda,\delta)$ so that for all sufficiently long and $H$--wide fillings $\pi : G \to \overline{G}$ the subgroup $\pi(H)$ is $\lambda'$--relatively quasiconvex in $\overline{G}$.
\end{proposition}
\begin{proof}
Recall that at the beginning of the section we fixed a constant $\delta$ so that the cusped space $X$ for $(G,\mc{P})$ is $\delta$--hyperbolic and that for sufficiently long fillings $\pi \co G \to \overline{G}$ the cusped space $\overline{X}$ for $(\overline{G},\overline{P})$ is also $\delta$--hyperbolic.  Suppose that $H$ is $\lambda$--relatively quasiconvex.

Let $\overline{h} \in \pi(H)$ and suppose that $h \in H$ is the shortest element of $H$ so that $\pi(h) = \overline{h}$.  Let $\gamma$ be a geodesic from $1$ to $h$ in $X$. By Corollary \ref{c:shortest} with $L = 10\delta$, for sufficiently long and  $H$--wide fillings
 there is a $10\delta$--local geodesic in $\overline{X}$ from $\overline{1}$ to $\overline{h}$ which lies in a $2$--neighborhood of $\pi(\gamma)$ and agrees with $\pi(\gamma)$ in a $10\delta$--neighborhood of the Cayley graph.

By Lemma \ref{l:local geodesic}, any geodesic from $\overline{1}$ to $\overline{h}$ is contained in an $(3\delta+2)$--neighborhood of $\pi(\gamma)$, and thus within a $(\lambda+3\delta+2)$--neighborhood of the image of the cusped space of $H$ in $\overline{X}$.  This suffices to prove the result, as in the proof of \cite[Proposition 4.3]{agm}.  (All that remains is to consider geodesics between points in the image of the cusped space of $H$ which do not lie at depth $0$, and it is straightforward to deal with these points given what has already been proved.)
\end{proof}

\begin{proposition} [cf. Proposition 4.4, \cite{agm}]\label{p:induced filling}
Let $H \le G$ be relatively quasiconvex.  For sufficiently long and $H$--wide fillings $\pi \co G \to \overline{G}$ the map from the induced filling of $H$ to $\overline{G}$ is injective.
\end{proposition}
\begin{proof}
Let $X$ be the cusped space for $G$ and $\overline{X}$ the cusped space for $\overline{G}$.  Suppose that $h \in H \cap \ker(\pi)$ is nontrivial.  Let $K_H$ be the kernel of the induced filling on $H$.  We must show that $h \in K_H$.

Let $\gamma$ be a geodesic in $X$ from $1$ to $h$, and note that $\pi(\gamma)$ is a loop.  Suppose that condition (1) from Lemma \ref{l:4.1} holds.  Then there is a nontrivial $10\delta$--local geodesic loop based at $1$ in $\overline{X}$ agreeing with $\pi(\gamma)$ in a $10\delta$--neighborhood of $1 \in \overline{X}$.  This is impossible.

Therefore, Lemma \ref{l:4.2} applies, and there is an element $k \in K_H$ so that $d_X(1,kh) < d_X(1,h)$.  Induction on the length of $h$ shows that $h \in K_H$, as required.
\end{proof}

\begin{proposition} [cf. Proposition 4.5, \cite{agm}]\label{p:weak sep}
Let $H \le G$ be relatively quasiconvex and suppose that $g \in G \smallsetminus H$.  For sufficiently long and $H$--wide fillings $\pi \co G \to \overline{G}$ we have $\pi(g) \not\in \pi(H)$.
\end{proposition}
\begin{proof}
Choose $L_1 = 3d_X(1,g) + 10\delta$ and any $L_2 > \max \{ 2L_1 + 3, R \}$, and suppose that $\pi$ is sufficiently long and $H$--wide that Lemmas \ref{l:4.1} and \ref{l:4.2} hold for $\pi$, and also so that $\pi$ induces a bijection between the ball of radius $L_1$ about $1$ in $X$ and the ball of radius $L_1$ about the image of $1$ in the cusped space of $\pi(G)$.

In order to obtain a contradiction, suppose that $\pi(g) \in \pi(H)$, and choose $h\in H\cap\pi^{-1}(\pi(g))$ with $d_X(1,h)$ minimal.  Let $\gamma$ be a geodesic from $1$ to $h$ in $X$, and let $\sigma$ be a geodesic from $1$ to $g$ in $X$.  Note that $\pi(\sigma)$ is a geodesic.

The minimality of $h$ and Lemma \ref{l:4.2} ensure that condition (1) from Lemma \ref{l:4.1} holds for $\gamma$. 

There are now two cases, depending on whether $\pi(\gamma)$ (equivalently $\gamma$) leaves the $L_1$--neighborhood of the Cayley graph.  
If $\gamma$ lies in the $L_1$--neighborhood of the Cayley graph, it is a $10\delta$--local geodesic joining $\overline{1}$ to $\overline{g}$.  Its length is therefore at most $\frac{7}{3}d_X(1,g)+2\delta<L_1$, by Lemma \ref{l:local geodesic}.  But since $\pi$ is injective on the $L_1$--ball about $1$, this implies $g=h$, a contradiction.

The second case is that $\pi(\gamma)$ leaves the $L_1$--neighborhood of the Cayley graph, in which case there is a $10\delta$--local geodesic as in Lemma \ref{l:4.1}, joining $\overline{1}$ to $\overline{g}$, which coincides with $\pi(\gamma)$ in the $L_1$--neighborhood of the Cayley graph, but may differ elsewhere.  The length of this $10\delta$--local geodesic is at least $L_1>\frac{7}{3}d_X(1,g)+2\delta$, again contradicting Lemma \ref{l:local geodesic}.
\end{proof}

We finish this section with an example which exhibits the necessity of restricting to $H$--wide fillings (and not just sufficiently long fillings) in Propositions \ref{p:pi(H) QC}, \ref{p:induced filling} and \ref{p:weak sep}.

\begin{example} \label{ex:long but not wide}
  Let $\Sigma$ be a genus $2$ surface, with $\pi_1\Sigma = F = \langle a,b,c,d\mid abcd(dcba)^{-1} \rangle$.  Let $\phi$ be an automorphism of $F$ induced by a pseudo-Anosov homeomorphism of $\Sigma$, so that the mapping torus $M_\phi$ has fundamental group 
  \[ G = \langle F,t\mid txt^{-1}=\phi(x),\mbox{ for }x\in F\rangle.\]
  By Thurston's geometrization of fiber bundles \cite{OtalSMF}, $M_\phi$ is a hyperbolic $3$--manifold; in particular $G$ is a hyperbolic group.
Now extend the centralizer of $a$ (attaching a torus to $M_\phi$ by gluing its longitude to a loop representing $a$) to get $\Gamma$:
\[ \Gamma = \langle G, e\mid [e,a] \rangle.\]
  Letting $P = \langle e,a\mid [e,a]\rangle$ we have a relatively hyperbolic pair $\left(\Gamma,\{ P \} \right)$, by \cite[Theorem 0.1.(2)]{Dahmani03}.

  Let $H = \langle b,c,d ,e \rangle < \Gamma$.  Then $H$ is a free group on the given generators.  This can be seen from the induced action of $H$ on the Bass--Serre tree of the defining graph of groups for $\Gamma$, which exhibits $H$ as the free product of the free group $\langle b,c,d \rangle < G$ and the infinite cyclic group $\langle e \rangle$.  The subgroup  $\langle b,c,d \rangle$ is quasiconvex in $G$, by a result of Scott and Swarup \cite{ScottSwarup90}.  It then follows from the argument in the proof of \cite[Proposition 4.6]{Dahmani03} that $H$ is relatively quasiconvex in $(\Gamma, \{ P \})$.

For an integer $i > 0$, let $N_i = \langle e^ia^{-1} \rangle \unlhd P$.  Taking the sequence of fillings
\[	\pi_i \co \Gamma \to \Gamma_i = \Gamma / \llangle N_i \rrangle	,	\]
gives a cofinal sequence of longer and longer Dehn fillings.  For each $i > 0$ the group $\Gamma_i$ has the following graph of groups decomposition:
\[	\Gamma_i = G \ast_{a = e^i} \langle e \rangle	.	\]
In particular, $G$ embeds in $\Gamma_i$ (as a quasiconvex subgroup).

The image $H_i := \pi_i(H)$ of $H$ in $\Gamma_i$ is an amalgam of $F$ with an infinite cyclic subgroup over a maximal cyclic subgroup of $F$.  Since $F$ is distorted in $G$, the subgroup $H_i$ is distorted in $\Gamma_i$, and hence is not quasiconvex.

Moreover, $a \not\in H$ but $\pi_i(a) \in H_i$ for all $i > 0$.  Finally, we have $N_i \cap H  = \{ 1 \}$, so the induced filling of $H$ is the trivial filling.  On the other hand, $H_i$ is not a free group, so the map from the induced filling of $H$ to $\Gamma_i$ is not injective for any $i > 0$.  Explicitly, the element $(e^i)bcd\left( dcb(e^i) \right)^{-1}$ is in the kernel of the map from the induced filling of $H$ to $\Gamma_i$.
\end{example}

\section{Existence of $H$--wide fillings}\label{s:H-wide exists}

In this section, we prove Lemma \ref{l:PF wide fillings} which implies that in our applications in Sections \ref{s:VF} and \ref{s:Henry} we can find sufficiently long and $H$--wide fillings.  The key observation is that separability  allows us to do this.

\begin{lemma} \label{l:simple wide}
Suppose that $P$ is a group and that $B$ is a separable subgroup.  For any finite set $S$ there exists a finite-index normal subgroup $K_S \le P$ so that for any $N \unlhd P$ with $N \le K_S$, the subgroup $N$ is $(B,S)$--wide in $P$.
\end{lemma}
\begin{proof}
  For each $s\in S\smallsetminus B$, choose some $P_s \le P$ finite index and satisfying $B<P_s$ and $s\not\in P_s$.  Let $K_S=\bigcap\{P_s \mid s\in S\smallsetminus B\}$, and note that $K_S$ is finite index in $P$, and contains $B$.

  Suppose $N\lhd P$ is contained in $K_S$.  We verify that $N$ is $(B,S)$--wide.  Let $b\in B$ and $s\in S$, and suppose $bs\in N$.  If $s\in B$ there is nothing to show, so suppose $s\not\in B$.  The element $s$ is not contained in $K_S$, but $bs\in N<K_S$, so we must have $b\not\in K_S$.  But this contradicts $B\le K_S$.

  The subgroup $K_S$ just constructed may not be normal, but we may replace $K_S$ by the intersection of its conjugates without disturbing the conclusion.
\end{proof}

In Lemma \ref{l:PF wide fillings} we consider a finite collection $\left\{ (H_1,\mc{D}_1) , \ldots , (H_k,\mc{D}_k) \right\}$ of relatively quasiconvex subgroups of a relatively hyperbolic pair $(G,\mc{P})$.  For each $i$ and each $D \in \mc{D}_i$ we assume that we have fixed $P_D \in \mc{P}$ and $c_D \in G$ so that $D \le P_D^{c_D}$, and that $c_D$ is a shortest such conjugating element.

\begin{lemma} \label{l:PF wide fillings}
Suppose that $(G,\mc{P})$ is relatively hyperbolic, and that $\mc{H}=\{(H_1,\mc{D}_1),\ldots,(H_k,\mc{D}_k)\}$ is a collection of relatively quasiconvex subgroups.  
Suppose that for each $1 \le j \le k$ and for each $D \in \mc{D}_j$ the subgroup $D$ is separable in $P_D^{c_D}$.

Then for any finite $S \subset \bigcup \mc{P} \smallsetminus \{ 1 \}$ there exist finite index subgroups $\{ K_P \unlhd P\mid P\in \mc{P} \}$ so that any filling
\[	G \to G(\mc{N})\mbox{, with }\mc{N} = \{N_P \le K_P\mid P\in \mc{P}\}	\]
is $(H_j,S)$--wide for each $1\le j\le k$.
\end{lemma}
\begin{proof}
Fix $S\subset \cup \mc{P} \smallsetminus \{ 1 \}$ a finite set.  

Fix $P\in \mc{P}$ and let $S_P = P\cap S$.
Suppose, for some $i$, that $D \in \mc{D}_i$
is so that $P = P_D$.  By Lemma \ref{l:simple wide} there is a finite-index normal subgroup $K_D \unlhd P$ so that any $N \unlhd P^{c_D}$ for which $N \le K_D^{c_D}$ is $(D,S_P)$--wide in $P^{c_D}$.

We choose $K_P$ to be the intersection of all $K_D$ for which $P = P_D$.  If we now choose $N_P \le K_P$ then the conclusion of the lemma holds.  This completes the proof.
\end{proof}

\section{Application to virtual specialness and virtual fibering}\label{s:VF}

In this section, we explain how the ideas and results in the beginning of the paper, together with a recent result of Cooper and Futer \cite{CooperFuter} give a proof of Theorem \ref{t:VS} independent of \cite{Wise}.

The following consequence of Theorem \ref{t:VS} was reproved (without using the results of \cite{Wise}) by Cooper and Futer.

\begin{theorem} \cite[Corollary 1.3]{CooperFuter} \label{t:CF}
 Suppose that $G$ is the fundamental group of a non-compact finite-volume hyperbolic $3$--manifold.  Then $G$ acts freely and cocompactly on a CAT$(0)$ cube complex dual to finitely many immersed quasi-Fuchsian surfaces.
\end{theorem}

In this section, our main result is that this cubulation is virtually special.

\begin{reptheorem}{t:VS}
Suppose that $G$ is the fundamental group of a non-compact finite-volume hyperbolic $3$--manifold $M$.  Then $G$ is virtually compact special.
\end{reptheorem}

If $\mc{P}$ is a collection of conjugacy-representatives of maximal parabolic subgroups of $G$ then $(G,\mc{P})$ is relatively hyperbolic.  After possibly replacing $M$ by an orientable double-cover of $M$, each element of $\mc{P}$ is free abelian of rank $2$.
As explained in the proof of Theorem \ref{t:VS} below, proving Theorem \ref{t:VS} reduces to establishing separability of certain double cosets of relatively quasiconvex subgroups of $G$. In the closed case, such double cosets are separable by results in \cite{VH} and \cite{minasyan:subsetgferf}.  We reduce to this case by performing orbifold Dehn filling on $M$ and applying the following `weak separability' criterion for double cosets.

\begin{proposition}\label{p:DC WS}
Suppose that $(G,\mc{P})$ is relatively hyperbolic and that each element of $\mc{P}$ is free abelian.  Suppose further that $\mc{H}$ is a finite collection of relatively quasiconvex subgroups of $(G,\mc{P})$ and that $S \subseteq (\bigcup \mc{P}) \smallsetminus \{ 1 \}$ and $F\subseteq G$ are finite subsets.

There exist finite-index subgroups $\{ K_P \unlhd P \mid P \in \mc{P} \}$ so that for any subgroups $N_P \le K_P$ the filling
\[	G \to G/K := G\left( \left\{ N_P  \mid P \in \mc{P} \right\} \right)	,\]
is $(H,S)$--wide for each $H \in \mc{H}$ and furthermore whenever $f \in F$, $\Psi,\Theta\in \mc{H}$ satisfy $1 \not\in \Psi\Theta f$, there is no element of $K$ in $\Psi \Theta f$.
\end{proposition}
\begin{proof}
By Lemma \ref{l:PF wide fillings} and the fact that all subgroups of finitely generated abelian groups are separable, there exist finite-index subgroups $K_P \unlhd P$ so that if $N_P \le K_P$ then the filling is $(H,S)$--wide for each $H$.  Below, we find other finite-index subgroups $\hat{K}_P \unlhd P$ so that if $N_P < \hat{K}_P$ then the condition on double cosets holds.  We then choose $N_P \le K_P \cap \hat{K}_P$.  For the remainder of the proof we concentrate on finding the subgroups $\hat{K}_P$.

Let $X$ be the cusped space for $(G,\mc{P})$ and suppose that $X$ is $\delta$--hyperbolic.  We further assume that $\delta$ is chosen so that the cusped spaces of all sufficiently long fillings are $\delta$--hyperbolic.   We suppose that $\delta \ge 1$.  Let $\lambda$ be a quasiconvexity constant which works for every element in $\mc{H}$.
Finally, let $M = \max\{d_X(1,f)\mid f\in F\}$.

Fix $f \in F$ and $\Psi, \Theta \in \mc{H}$ so that $1 \not\in \Psi\Theta f$, and consider the equation $k \in \Psi\Theta f$ for elements $k$ of the kernel of a filling.  After finding conditions on the filling which ensure there is no such element, we consider a filling appropriate for all $f \in F$ simultaneously.

  Let $\mc{D}$ be a collection of representatives of $\Psi$--conjugacy classes of maximal uniquely parabolic subgroups of $\Psi$.  For $D \in \mc{D}$, we have $D \le P_{D}^{c_D}$ for some $P_{D} \in \mc{P}$ and some (shortest) $c_D \in G$.  Similarly, let $\mc{E}$ be a collection of representatives of $\Theta$--conjugacy classes of maximal parabolic subgroups of $\Theta$, and for $E \in \mc{E}$ we have $E \le P_E^{d_E}$ for some $P_E \in \mc{P}$ and some (shortest) $d_E \in G$.  Let $X_\Psi$ be the cusped space for the pair $(\Psi,\mc{D})$ and let $X_\Theta$ be the cusped space for $(\Theta,\mc{E})$ (both with respect to some choices of generating sets).  Let $\check{\iota}_\Psi \co X_\Psi \to X$ and $\check{\iota}_\Theta \co X_\Theta \to X$ be the induced maps of cusped spaces, and note that $\check{\iota}_\Psi(X_\Psi^{(0)})$ and $\check{\iota}_\Theta(X_\Theta^{(0)})$ are both $\lambda$--quasiconvex subsets of $X$.

  In order to apply the results from Section \ref{s:Properties}, choose $L_1 = 10\delta$ and $L_2 = \max \{ 20\delta +  M  + \lambda+4, R_\Psi, R_\Theta \}$, where $R_\Psi$ and $R_\Theta$ are the constants from Proposition \ref{prop:MMP A.6} applied to $\Psi$ and $\Theta$, respectively.

For $P \in \mc{P}$, let
\[ S_P \supseteq \{p\in P\mid d_X(1,p)\leq 32\delta + 2 M  + 4\lambda + 2L_1  + 3\} \]
be a finite set which is large enough so that for all $H \in \mc{H}$ the $(H,S)$--wideness condition of Lemma \ref{l:4.2} is satisfied with $L_1$ and $L_2$ as above.
Consider the collection of subgroups of $P$ of the form $P_{B_1,B_2} = \langle B_1,B_2 \rangle$ where $B_1 = P \cap \Psi^{g_1}$ for some $g_1 \in G$ and $B_2 = P \cap \Theta^{g_2}$ for some $g_2 \in G$.  There are finitely many such pairs of subgroups of $P$.  
For each such pair $(B_1,B_2)$, 
by Lemma \ref{l:simple wide} there exists a finite-index $\hat{K}_{B_1,B_2} \lhd P$ which is $(P_{B_1,B_2},S_P)$--wide.
 We define $\hat{K}_P = \bigcap \hat{K}_{B_1,B_2}$ and check the condition on double cosets.

Choose $N_P \le \hat{K}_P$, and consider the filling
\[	G \to G\left( \left\{ N_P  \mid P \in \mc{P} \right\}  \right) = G/K	.	\]

In order to obtain a contradiction suppose that there is an element $g \in K$, and  elements $f\in F$, $\psi \in \Psi$ and $\theta \in \Theta$ so that
\[	g = \psi\theta f	.	\]
Choose a $g$ so that $d_X(1,g)$ is minimal amongst all choices of $g$ for which there is such an expression.  

Consider a geodesic quadrilateral in $X$ with vertices $1,\psi,\psi\theta,g$, and let $\xi_1$ be the geodesic from $1$ to $\psi$, $\xi_2$ the geodesic from $\psi$ to $\psi\theta$, $\eta$ the geodesic from $\psi\theta$ to $g$ and $\rho$ the geodesic from $1$ to $g$, respectively.  By assumption, we know that $g \ne 1$.

Let $\pi \co X \to \overline{X}$ be the map on cusped spaces induced by the filling map $\pi \co G \to G/K$.
Since $g \in K \smallsetminus \{ 1 \}$ the image of $\rho$ in $\overline{X}$ is a loop, so condition (1) from Lemma \ref{l:4.1} cannot hold.  This means that condition (2) from Lemma \ref{l:4.1} holds.  Let $A$ be a horoball $L_2$--penetrated by $\rho$, so $\rho$ meets $A$ in a segment $[x,y]$, and let $k \in K \cap \Stab(A)$ be so that $d_X(x,k.y) \le 2L_1+3$. 
It is straightforward to see that $d_X(1,k.g) < d_X(1,g)$, since $d_X(x,k.y) \le 2L_1+3$ but $d_X(x,y) \ge 2L_2 > 2L_1 + 4$.  We arrive at a contradiction by showing that $k.g \in \Psi\Theta f$, contradicting the choice of $g$ as a shortest element of $K$ with such an expression.

Without loss of generality, the subsegment of $\rho$ between $x$ and $y$ is a geodesic through $A$ which consists of a vertical segment down from $x$, a horizontal segment of at most $3$ edges, and then a vertical segment terminating at $y$ (see \cite[Lemma 3.10]{rhds}).  Let $x'$ be the point  on this geodesic directly below $x$ at depth $3\delta +  M  + \lambda$ and let $y'$ be the point directly below $y$ at depth $3\delta +  M  + \lambda$.  The quadrilateral $\xi_1\cup \xi_2 \cup \eta \cup \rho$ is $2\delta$--slim, so there are points on $\eta \cup \xi_1 \cup \xi_2$ within $2\delta$ of $x'$ and of $y'$.
Because $\eta$ is a geodesic (of length at most $ M $) joining two points at depth $0$, no point on $\eta$ can be within $2\delta$ of either $x'$ or $y'$.  Therefore, there are points on $\xi_1 \cup \xi_2$ within $2\delta$ of $x'$ and of $y'$.

The geodesic $\xi_1$ travels between two points in $\Psi$ and $\xi_2$ joins two points in $\psi\Theta$.  By quasiconvexity, any point on $\xi_1$ lies within $\lambda$ of a point in $\check{\iota}_\Psi(X_\Psi)$ and any point on $\xi_2$ lies within $\lambda$ of a point in $\psi . \check{\iota}_\Theta(X_\Theta)$.  Let $u_0$ and $v_0$ be points in $\check{\iota}_\Psi(X_\Psi) \cup \psi. \check{\iota}_\Theta(X_\Theta)$ lying within distance $2\delta + \lambda$ of $x'$ and $y'$ respectively.  Note that $u_0$ and $v_0$ lie in the horoball $A$.

There are points at depth more than $0$ in a horoball $A$ which lie in $\check{\iota}_\Psi(X_\Psi)$ exactly when they are of the form $(sc_DP_D, hc_D,n)$ for some $s \in \Psi$, $h \in sD$ and $n \in \N$, where $D \in \mc{D}$ is so that $D \le P_D^{c_D}$, as above, and $A$ is the horoball based on the coset $sc_DP_D$.

Similarly, there are points at depth more than $0$ in $A$ which lie in $\psi . \check{\iota}_\Theta(X_\Theta)$ exactly when they are of the form $(\psi.td_EP_E, \psi. gc_D, m)$ for $t \in \Theta$, $g \in tE$ and $m \in \N$, where $E \in \mc{E}$ and $E \le P_E^{d_E}$, and $A$ is the horoball based on $\psi.td_EP_E$.

The points $u_0$, $v_0$ have one of these forms, and they are at distance at most $2\delta + \lambda$ from $x'$ and $y'$ respectively, which implies that the appropriate $n$ or $m$ is at most $3\delta +  M  + 2\lambda$.  Thus, there are points $u, v$ at depth $0$ in $A$, directly above $u_0$ and $v_0$ respectively, so that $d_X(u,x), d_X(v,y) \le 10\delta + 2 M  + 4\lambda$.  All of the points in $A$ directly above $u_0$ lie in $\check{\iota}_\Psi(X_\Psi)$ or $\psi. \check{\iota}_\Theta(X_\Theta)$, except possibly the point $u$ at depth $0$.  This point $u$ will not lie in $\check{\iota}_\Psi(X_\Psi)$ unless $c_D = 1$, and similarly for $\check{\iota}_\Theta(X_\Theta)$.  However, certainly $u$ lies within distance $1$ of $\check{\iota}_\Psi(X_\Psi)$ or $\psi. \check{\iota}_\Theta(X_\Theta)$.  Similarly, $v$ lies within distance $1$ of $\check{\iota}_\Psi(X_\Psi)$ or $\psi. \check{\iota}_\Theta(X_\Theta)$.

We deal with four cases, depending on whether each of $u_0$ and $v_0$ are contained in $\check{\iota}_\Psi(X_\Psi)$ or $\psi . \check{\iota}_\Theta(X_\Theta)$.

{\bf Case 1:}  Both $u_0$ and $v_0$ are contained in $\check{\iota}_\Psi(X_\Psi)$.

(The case where they are both contained in $\psi. \check{\iota}_\Theta(X_\Theta)$ is entirely similar and we omit it.)

In this case, if $u_0 = (sc_DP_D,\psi_uc_D,n)$ then $u = \psi_uc_D$, where $\psi_u \in \Psi$ and $A$ is the horoball based on $sc_DP_D$.  For ease of notation we write $P = P_D$ and $c = c_D$, and so we have $\psi_uc \in scP$.

Note that $v = \psi_vc \in scP$ also, for some $\psi_v \in \Psi$.
We have $u^{-1}v = c^{-1}(\psi_u^{-1}\psi_v)c \in D^{c^{-1}}$.

Now, $d_X(x,k.y) \le 2L_1+3$, and we also have $d_X(x,u), d_X(y,v) \le \alpha$, from which it follows that $d_X(u,k.v) \le 2\alpha + 2L_1+3$.  Write $w = v^{-1}k^{-1}u$, a group element of $X$--length at most $2\alpha + 2L_1+3$ and note that $u^{-1}k^{-1}u$ is in the filling kernel $N_P \lhd P$, and so $cu^{-1}k^{-1}uc^{-1}$ is contained in $N_P^c$.  On the other hand, we also have
\begin{eqnarray*}
cu^{-1}k^{-1}uc^{-1} &=&  cu^{-1}v(v^{-1}k^{-1}u)c^{-1} \\
&=& cu^{-1}vwc^{-1} \\
&=& c (c^{-1}\psi_u^{-1}\psi_v c)wc^{-1} \\
&=& (\psi_u^{-1}\psi_v) cwc^{-1}	.
\end{eqnarray*}	
Note that $w\in S_P$, and that the filling is $(\Psi,S_P)$--wide.
Since $(\psi_u^{-1}\psi_v) \in D$, this implies that $cwc^{-1} \in D$, so $cu^{-1}k^{-1}uc^{-1} = (\psi_u^{-1}\psi_v)cwc^{-1} \in D$.
Therefore, $\psi_u^{-1}k\psi_u = cu^{-1}kuc^{-1} \in D$, which means that 
\[	k.\psi = k\psi_u(\psi_u^{-1}\psi) = \psi_u (\psi_u^{-1}k\psi_u)\psi_u^{-1}\psi \in \Psi	.	\]
Therefore,
\[	k.g = \left(k.\psi \right) \theta f	,	\]
gives an expression for $k.g$ as an element of $\Psi\Theta f$, contradicting the fact that $g$ was the shortest element of $K$ with such an expression.

{\bf Case 2:}  $u_0$ is contained in $\check{\iota}_\Psi(X_\Psi)$ and $v_0$ contained in $\psi .\check{\iota}_\Theta(X_\Theta)$.

We can write $u_0 = (sc_DP_D,\psi_uc_D,n)$ and $v_0 = (\psi.td_EP_E, \psi. r_vd_E, m)$, so $u = \psi_uc_D$ and
$v = \psi.\theta_v d_E$, where $\psi_u \in \Psi$, $\theta_v \in \Theta$, $D \le P_D^{c_D}$ and $E \le P_E^{d_E}$.
We clearly have $P_E = P_D$, which we write as $P$.  We write $c = c_D$ and $d = d_E$, and note that 
$A$ is the horoball based on the coset $scP = \psi tdP$.

We still have $d_X(u,x), d_X(v,y) \le \alpha$.  The geodesic $\xi_1$ intersects $A$ in a segment $[g_1,h_1]$, where the entrance point $g_1$ is within $4\delta$ of $x$, and within $\alpha$ of $u$.  The exit point $h_1$, we may similarly argue, is within $\alpha$ of some group element $w = \psi_w c$ in the coset $scP$, with $\psi_w\in \Psi$.  Likewise, the geodesic $\xi_2$ intersects the horoball $A$ in a segment $[g_2,h_2]$, where $d_X(g_2,h_1)\leq 4\delta$, and there is another point $z = \psi\theta_z d$ in $scP$ with $\theta_z\in \Theta$, and satisfying $d_X(z,g_2)\leq \alpha$.  See Figure \ref{fig:case2}.
\begin{figure}[htbp]
  \centering
  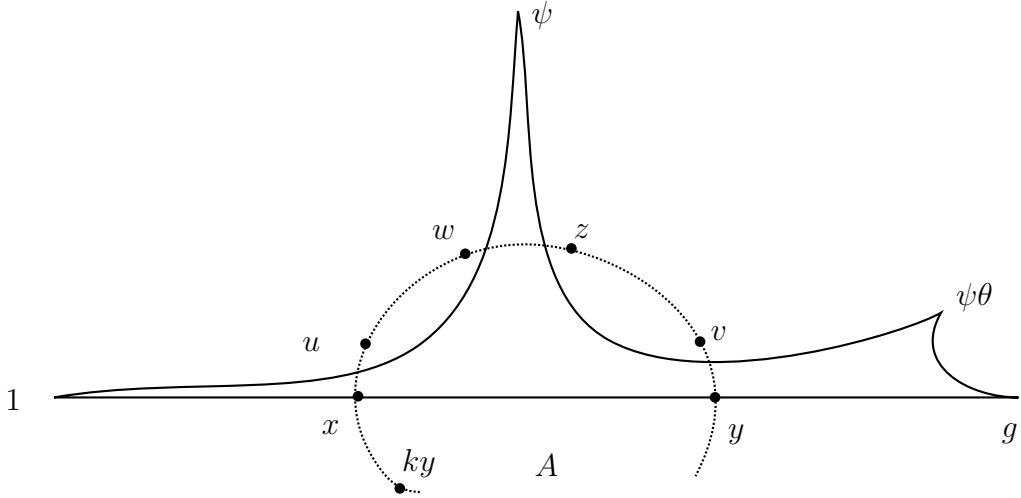
  \caption{Case 2.  The dotted line represents the coset $scP = \psi tdP$.}
  \label{fig:case2}
\end{figure}

We have $u^{-1}w \in D^{c^{-1}}$ and
$z^{-1}v \in E^{d^{-1}}$.  Both $D^{c^{-1}}$ and $E^{d^{-1}}$ are subgroups of $P$.

Let $B_1 = D^{c^{-1}}$ and $B_2 = E^{d^{-1}}$ so $u^{-1}wz^{-1}v \in P_{B_1,B_2}$.  Now,
\begin{eqnarray*}
u^{-1}k^{-1}u &=& (u^{-1}w) (w^{-1}z) (z^{-1}v) (v^{-1}k^{-1}u) \\
&=& (u^{-1}w)(z^{-1}v) \left( (w^{-1}z) (v^{-1}k^{-1}u) \right)	.
\end{eqnarray*}
(Note that we use here that $P$ is abelian.)

Since $w^{-1}z$ has length at most $2\alpha + 4\delta$ and $v^{-1}k^{-1}u$ has length at most $2\alpha + 2L_1+3$, so the last of the three terms above is in $S_P$ and
 we can apply the $(P_{B_1,B_2},S_P)$--wideness of the kernel to deduce that
$u^{-1}ku \in P_{B_1,B_2} = B_1B_2$.  Choose elements $b_1 \in B_1$ and $b_2 \in B_2$ so that $u^{-1}ku = b_1b_2$.  We now have
\begin{eqnarray*}
k \psi \theta 
&=& u (u^{-1}ku) (u^{-1}w) (w^{-1}z) (z^{-1}v) (v^{-1}\psi \theta) \\
&=& u (b_1b_2) (u^{-1}w) (w^{-1}z) (z^{-1}v) (v^{-1} \psi \theta) \\
&=& u \left( b_1(u^{-1}w) \right) (w^{-1}z) \left( (z^{-1}v)b_2 \right) (v^{-1}\psi \theta) \\
&=& u \left( b_1(u^{-1}w) \right) (w^{-1} \psi) (\psi^{-1}z) \left( (z^{-1}v)b_2 \right) (v^{-1}\psi \theta) \\
&=& \psi_u \left( c (b_1u^{-1}w) c^{-1} \right) (cw^{-1}\psi) (\psi^{-1}z) \left( (z^{-1}v) b_2\right) (v^{-1}\psi \theta) \\
&=&\left(  \psi_u \left( c (b_1u^{-1}w) c^{-1} \right) (\psi_w^{-1} \psi) \right) \left( \theta_z \left(d (z^{-1}v b_2) d^{-1}\right) (\theta_v^{-1} \theta) \right) \\
\end{eqnarray*}
The first three terms of this expression are in $\Psi$ and the last three terms are in $\Theta$, which proves
that $k\psi \theta \in \Psi \Theta$.   Therefore, $k.g = k \psi \theta f \in \Psi \Theta f$.  Since we know that $d_X(1,k.g) < d_X(1,g)$, this contradicts the minimality of $g$, hence proving the result in Case 2.

It remains to note that the case that $u_0$ is contained in $\psi . \check{\iota}_\Theta(X_\Theta)$ and $v_0$ is contained in $\check{\iota}_\Psi(X_\Psi)$ essentially becomes Case 1.  Indeed, suppose that $v_0$ is contained in $\check{\iota}_\Psi(X_\Psi)$.  Then the geodesic from $v$ to $1$ lies near to the geodesic from $y$ to $1$, which easily implies (since $x$ lies on the geodesic from $y$ to $1$) that $x$ lies near $\check\iota_\Psi(X_\Psi)$, as required.  This completes the proof of Proposition \ref{p:DC WS}.
\end{proof}

We now turn to the proof of Theorem \ref{t:VS}.

\begin{proof}[Proof (of Theorem \ref{t:VS})]
Pass to a finite cover which is orientable, so that all cusps in $M$ have torus cross-sections.
It is well known that if $\mc{P}$ is a collection of representatives of $G$--conjugacy classes of maximal cusp subgroups of $M$ then $(G,\mc{P})$ is relatively hyperbolic.

By Theorem \ref{t:CF}, there is a CAT$(0)$ cube complex $X$ upon which $G$ acts freely and cocompactly.
By \cite[Criterion 2.3]{PW} (see also \cite[Section 4]{HaglundWise-Coxeter}), to prove that the action of $G$ on $X$ is virtually special it suffices to prove that for a certain finite list of subgroups $Q_i$ which stabilize hyperplanes in $X$, the subgroups $Q_i$ and double cosets $Q_iQ_j$ are separable in $G$.  Since $G$ is a Kleinian group, it is LERF by \cite[Corollary 9.4]{VH}, so it remains to prove double coset separability.  

The cube complex $X$ built by Cooper and Futer for Theorem \ref{t:CF} is built using the Sageev construction \cite{Sageev} (see also \cite{hruskawise:finiteness}).  The hyperplane subgroups in $G$ are commensurable to the codimension $1$ subgroups, which are quasi-Fuchsian surface subgroups and therefore geometrically finite.  It now follows immediately by \cite[Corollary 1.3]{HruskaQC} that the subgroups $Q_i$ are relatively quasiconvex in $G$ .
  
Suppose now that $h \not\in Q_iQ_j$.  Equivalently, $1 \not\in Q_iQ_jh^{-1}$. By Proposition \ref{p:pi(H) QC}, for sufficiently long and $Q_i$--wide fillings the image of $Q_i$ is relatively quasiconvex, and similarly for $Q_j$.
 By Proposition \ref{p:DC WS}, there exist finite-index subgroups $\{ K_P \unlhd P \mid P \in \mc{P} \}$ so that for any choices $\{ \gamma_P \in K_P \mid P \in \mc{P} \}$, the filling
\[ G \to \overline{G} = G\left( \left\{ \gamma_P \mid P \in \mc{P} \right\} \right)		\]
is such that the images of $Q_i$ and $Q_j$ are relatively quasiconvex and there is no element of $K$ in $Q_iQ_j h^{-1}$.  For such a filling, the image of $h$ is outside the image of $Q_iQ_j$.  
 
 Possibly replacing the $\gamma_P$ by powers (which does not change containment in $K_P$, and so the above properties continue to hold), the Orbifold Hyperbolic Dehn Surgery Theorem \cite[Theorem 5.3]{DunbarMeyerhoff}, implies that the group $\overline{G}$ is the fundamental group of a compact hyperbolic orbifold, and so is Kleinian and word-hyperbolic.  Since it is Kleinian, it is LERF by \cite[Corollary 9.4]{VH}, and since it is also word-hyperbolic \cite[Theorem 1.1]{minasyan:subsetgferf} implies that all double cosets of quasiconvex subgroups of $\overline{G}$ are separable.  Therefore, the image of $h$ can be separated from the image of $Q_iQ_j$ in a finite quotient of $\overline{G}$, which is clearly also a finite quotient of $G$.

This proves that $Q_iQ_j$ is separable in $G$, which proves that the $G$--action on $X$ is virtually special, as required.
\end{proof}

\section{Relative height and relative multiplicity}\label{s:Henry}

For quasiconvex subgroups of hyperbolic groups, the {\em height} is an important invariant.  For full relatively quasiconvex subgroups, it remains a useful invariant, but because we cannot control the normalizer in $P$ of an intersection $H \cap P$ when $H$ is (non-full) relatively quasiconvex and $P$ is a maximal parabolic subgroup, height is not always a useful notion as it is too often infinite.  Instead, we should consider the relative height, defined as follows.

\begin{definition} (cf. \cite[$\S$1.4]{hruskawise:packing})
Suppose that $(G,\mc{P})$ is relatively hyperbolic and $H \le G$.  The {\em relative height} of $H$ in $(G,\mc{P})$ is the maximum number $n \ge 0$ so that there are distinct cosets $\{ g_1H ,\ldots , g_nH \}$ so that
$\bigcap\limits_{i=1}^n g_iHg_i^{-1}$ is an infinite non-parabolic subgroup.
\end{definition}
In \cite{hruskawise:packing}, they refer to relative height merely as `height', but we prefer to keep this term for its traditional meaning.

\begin{remark}
It follows from the classification of groups acting isometrically on $\delta$--hyperbolic spaces that a subgroup of a relatively hyperbolic group is infinite and non-parabolic if and only if it contains a loxodromic element.  We use this equivalent characterization without further mention.
\end{remark}

In this section, we prove results for relative height analogous to those proved for height in \cite[Appendix A]{VH}.  Specifically, we define a notion of {\em relative multiplicity} (see Definition \ref{d:rel mul}) and prove in Theorem \ref{t:rh is rm} that relative multiplicity is equal to relative height.  This gives a new proof of a theorem of Hruska and Wise \cite{hruskawise:packing} that the relative height of a relatively quasiconvex subgroup is finite.
  In Theorem \ref{t:rh decrease} we prove that for sufficiently long and $H$--wide fillings the relative height of a relatively quasiconvex subgroup does not increase under Dehn filling.

The definition of a \emph{weakly geometrically finite} (or \emph{WGF}) action is given in \cite[A.27]{VH}.  We note here that a weakly geometrically finite action differs from the usual notion of a \emph{geometrically finite} action (as in \cite[Definition 3.2 (RH-2)]{HruskaQC}) in allowing horoballs with finite stabilizer.

Fix a relatively quasiconvex subgroup $(H,\mc{D})$.  
Let $X_H$ be a cusped space for $(H,\mc{D})$ and  $\check\iota \co X_H \to X$ be the extension of the natural inclusion of $H$ into $G$ on the level of cusped spaces, as described in \cite[Section 3]{agm}.
\begin{definition}\label{def:RHull}
Suppose that $(G,\mc{P})$ is relatively hyperbolic and that $(H,\mc{D})$ is relatively quasiconvex.  Let $X$ be a cusped space for $(G,\mc{P})$ (considered as containing $G$ as a subset), let $\widetilde{\ast}$ be the basepoint of $X$, and let $R \ge 0$.  An {\em $R$--hull for $H$ acting on $X$} is a connected $H$--invariant full sub-graph $\widetilde{Z} \subset X$ so that
\begin{enumerate}
\item\label{RHull:basepoint} $\widetilde{\ast} \in \widetilde{Z}$;
\item\label{RHull:geodesic} If $\gamma$ is a geodesic in $X$ with endpoints in $\Lambda(H)$ then $N_R(\gamma) \cap N_R(G) \subset \widetilde{Z}$; and
\item\label{RHull:horoball} If $A$ is any horoball containing a vertex $a$ of depth greater than $0$ in the image $\check\iota(X_H)$, then $\tilde{Z}\cap A^{(0)}$ contains every vertex of a maximal vertical ray in $A$ containing $a$.
\item\label{RHull:WGH} The action of $(H,\mc{D})$ on $\widetilde{Z}$ (with its induced path metric) is WGF.
\end{enumerate}
\end{definition}

\begin{remark}
 This definition is not the same as \cite[Definition A.32]{VH} unless $H$ is full relatively quasiconvex (as was assumed in \cite{VH}).  It is important that we do not include an $R$--neighborhood of $\gamma$, but only that part of the $R$--neighborhood near the Cayley graph.  The third condition in \cite[Definition A.32]{VH} has similarly been modified.
Both of these changes are made so that Lemma \ref{l:embed} below is true.
\end{remark}

\begin{definition}
 Suppose that $(H,\mc{D}) \le (G,\mc{P})$ is relatively quasiconvex, and $\check{\iota}\co X_H\to X$ is the inclusion of cusped spaces as above.  For a positive integer $D$, the {\em restricted $D$--neighborhood of $\check{\iota}(X_H)$}, denoted $N^R_D(\check{\iota}(X_H))$, is the full subgraph of $X$ on the vertices of either of the following two types:
\begin{enumerate}
 \item Vertices of $N_D(\check{\iota}(X_H))\cap N_D(G)$; and
 \item For any horoball $A$ so $\check{\iota}(X_H)$ contains vertices of arbitrary depth in $A$, include all vertices $a\in A$ which are connected by a vertical geodesic to a vertex of the first type.
\end{enumerate}
\end{definition}

\begin{lemma} [cf. Lemma A.41,  \cite{VH}]
Let $R \ge 0$.  There exists some $D$ so that the restricted $D$-neighborhood of $\check\iota(X_H)$ is an $R$--hull for the action of $H$ on $X$.
\end{lemma}
\begin{proof}
  We first note that if any of the requirements of an $R$--hull are satisfied by the restricted $D$--neighborhood of $\check\iota(X_H)$, then they are satisfied for the restricted $D'$--neighborhood of $\check\iota(X_H)$, for any $D'\geq D$.  It therefore suffices to consider each of the four requirements separately, and take $D$ to be the maximum needed for any of the four.

  Condition \eqref{RHull:basepoint} is satisfied for any $D$, since $\widetilde{\ast} = 1\in H$.

  Condition \eqref{RHull:geodesic} is satisfied as soon as $D\geq R+2\delta+\lambda$, where $\lambda$ is the quasiconvexity constant for $\check\iota(X_H)$.  Indeed, $\Lambda(H) \subset \Lambda(\check\iota(X_H))$, so if $\gamma$ is a biinfinite geodesic with endpoints in $\Lambda(H)$, it must lie in a $\lambda+2\delta$--neighborhood of $\check\iota(X_H)$.  Suppose $x\in N_R(\gamma)\cap N_R(G)$; we want to show that $x\in N^R_D(\check{\iota}(X_H))$.  Let $z\in \gamma$, $g\in G$ be vertices at distance at most $R$ from $x$.  As we have noted, there is a $q\in \check\iota(X_H)$ satisfying $d_X(z,q)\leq \lambda+2\delta$.  Thus $x$ lies in the $(R+\lambda+2\delta)$--neighborhood of $\check\iota(X_H)$.  If $D\geq R+2\delta+\lambda$, then $x\in N_D(\check\iota(X_H))\cap N_D(G)$.

  Condition \eqref{RHull:horoball} is built in to the definition of restricted $D$--hull.

  Condition \eqref{RHull:WGH} (the weak geometric finiteness) follows once we observe that for large enough $D$, the restricted $D$--neighborhood is equivariantly quasi-isometric to the $D$--neighborhood, and either one is quasi-isometrically embedded in $X$.  In particular, the limit set of the restricted $D$--neighborhood is equivariantly homeomorphic to $\partial X_H =\partial(H,\mc{D})$.  (Though in general $\check\iota$ is not a quasi-isometric embedding if some $D \in \mc{D}$ is very distorted in $P_D^{c_D}$.)
\end{proof}

Let $\widetilde{Z}$ be an $R$--hull for the action of $H$ on $G$, and let $Z = \leftQ{\widetilde{Z}}{H}$.  Similarly, let $Y = \leftQ{X}{G}$.  Then there is a natural map $i \co Z \to Y$ which induces the inclusion of $H$ into $G$ (in the sense described in \cite{VH}).

For $n > 0$, let
\[  S_n = \left\{(z_1,\ldots,z_n)\in Z^n\mid
i(z_1)=\cdots=i(z_n)\right\}\setminus\Delta \]
where $\Delta = \left\{ (z_1, \ldots , z_n) \mid \mbox{ there exist $i \ne j$ so that } z_i = z_j \right\}$ is the fat diagonal.

Points in $S_n$ have a well-defined {\em depth} which is the depth of the image in $Y$
We consider components $C$ of $S_n$ which contain a point with depth $0$.

As in \cite{VH}, choosing a maximal tree in $Z$, and a basepoint $p$ at depth $0$, a component $C$ of $S_n$ induces well-defined maps $\tau_{C,i} \co \pi_1(C,p) \to H$, for $i = 1, \ldots , n$.

\begin{definition} \label{d:rel mul}
  The {\em relative multiplicity} of $i \co Z \to Y$ is the largest $n$ so that $S_n$ contains a component $C$ so that for all $i \in 1, \ldots , n$ the group
$	\tau_{C,i}(\pi_1(C,p))	$
  contains a loxodromic element.
\end{definition}

\begin{theorem}[cf. Theorem A.38, \cite{VH}]\label{t:rh is rm}
For sufficiently large $R$, depending only on $\delta$ and the quasi-convexity constant of $H$, if $\widetilde{Z}$ is an $R$--hull for the action of $H$ on $X$, and $i : Z \to Y$ is as described above, then the relative height of $H$ in $G$ is equal to the relative multiplicity of $i : Z \to Y$.
\end{theorem}
\begin{definition}
  A geodesic $\sigma$ in a combinatorial horoball is \emph{regular} if it has at most three horizontal edges, and these are at the maximum depth for $\sigma$.  A path in a cusped space $X(G,\mc{P})$ is \emph{regular} if every intersection with a horoball is regular.

  A path $\sigma$ in a combinatorial horoball is \emph{super-regular} if it has at most $1$ horizontal edge, this edge is at maximum depth for $\sigma$, and $\sigma$ has minimal length among paths with this property.  A path in a cusped space $X(G,\mc{P})$ is super-regular if every intersection with a horoball is super-regular.
\end{definition}

\begin{lemma}\label{l:quasiaxis}
  Let $g$ be a loxodromic element of the relatively hyperbolic group pair $(G,\mc{P})$.  Then for any $D>0$, and any sufficiently large $n>0$, there is a bi-infinite quasigeodesic axis $\sigma$ for $g^n$ satisfying:
  \begin{enumerate}
  \item $\sigma$ is super-regular;
  \item $\sigma$ is contained in a $(4\delta + 3)$--neighborhood of any geodesic with the same endpoints.
  \item $d_X(p,g^np)>D$ for any point $p\in \sigma$.
  \end{enumerate}
\end{lemma}
\begin{proof}
  Let $g^{\pm \infty}$ be the two limit points in $\partial X$ of the cyclic group $\langle g\rangle$.  Since $X$ is proper, there is a bi-infinite geodesic $\gamma$ joining $g^{\pm \infty}$.  Note that $g^n\gamma$ and $\gamma$ are Hausdorff distance at most $2\delta$ from one another, for any $n$.  Fix $n$ large enough so that $d_X(x,g^n x)>\max\{D,100\delta\}$ for every point $x\in X$.

Since the endpoints of $\gamma$ are distinct, $\gamma$ is not contained in a single horoball.  
Choose some $h\in \gamma$ in the Cayley graph of $G$.    Choose a regular geodesic $\alpha_0$ joining $h$ to $g^n h$.  

Let $\sigma_0$ be the concatenation of the $g^n$--translates of $\alpha_0$; namely $\sigma_0 = \bigcup_{i\in \bZ}g^{in}\alpha_0$.  Let $\sigma$ be the path obtained my modifying $\sigma_0$ to be super-regular.  (This means first ensuring that paths within any horoball consist of two vertical segments and a single horizontal segment, and then removing the horizontal subsegments of $\sigma_0$ inside horoballs, and replacing them by minimal length super-regular paths with the same endpoints.)

The path $\sigma_0$ is a broken geodesic with each breakpoint on $g^k\gamma$ for some $\gamma$.  The individual geodesics have length at least $100\delta$, and the Gromov products at the vertices are at most $6\delta$.  In particular, $\sigma$ is a quasi-geodesic.  The local modifications producing $\sigma$ from $\sigma_0$ do not change the fact of quasi-geodesicity (though they do change the constants of quasi-geodesicity).

The path $\sigma_0$ lies a $2\delta$--neighborhood of $\gamma$.  The path $\sigma$ thus lies in a $(2\delta+3)$--neighborhood of $\gamma$, and in a $(4\delta+3)$--neighborhood of any other geodesic with the same endpoints.
\end{proof}

\begin{proof}[Proof of Theorem \ref{t:rh is rm}]
The proof from \cite{VH} works almost as written.  
Let $\lambda$ be the constant of quasiconvexity for $\check\iota(X_H)$, and let $C = 2(\lambda+2\delta)+\max_i\{d_X(1,c_D)\}$, where the elements $c_D$ are those elements chosen as in Section \ref{s:background}.
We suppose $R > C + \lambda + 6\delta + 4$.

We first show the more difficult direction, that relative multiplicity dominates relative height.  Suppose that the relative height is at least $n$, so there is some collection of cosets $\{H,g_2H,\ldots,g_nH\}$ and loxodromic elements $h_1,\ldots h_n\in H$ so that $h_1=g_2h_2g_2^{-1}=\cdots=g_nh_ng_n^{-1}$.  
Let $\sigma$ be the quasi-axis for $h_1$ given by Lemma \ref{l:quasiaxis}, and let $\gamma$ be any bi-infinite geodesic with the same endpoints at infinity as $\sigma$.  Requirement \eqref{RHull:geodesic} implies that $N_R(\gamma)\cap N_R(G)$ is contained in $J=\widetilde{Z}\cap g_2\widetilde{Z}\cap \cdots \cap g_n\widetilde{Z}$.  In particular, any points of $\sigma\cap N_{R-(4\delta+3)}(G)$ are contained in $J$.  We next need to show that the deeper points of $\sigma$ are also contained in $J$. 

Choose $g\in G$ on the quasi-axis $\sigma$.  Fix $1 \le i \le n$.  We claim that there is an element $\hat{h}_i$ in $H$ so that $d_X(g,g_i\hat{h}_i) \le C$ (taking $g_1 = 1$).  Indeed,  $g_ih_ig_i^{-1} = h_1$ leaves $\sigma$ invariant, so the endpoints of $\sigma$ lie in $g_i\Lambda_H$.  Suppose that $\rho$ is a bi-infinite geodesic with the same endpoints as $\sigma$.  Then by Lemma \ref{l:quasiaxis} $\sigma$ lies in a $(4\delta+3)$--neighborhood of $\rho$.  On the other hand, quasi-convexity implies that any point on $\rho$ lies within distance $\lambda + 2\delta$ of $\check\iota(X_H)$.  Possibly, the point $g$ lies within $\lambda + 2\delta$ of a point in $\check\iota(X_H)$ which lies within a horoball, but then this point has depth at most $\lambda + 2\delta$, and so lies within distance $\lambda+2\delta + \max\{ d_X(1,c_D) \}$ of a point in $H$.
The claim follows.
\begin{figure}[htbp]
  \centering
  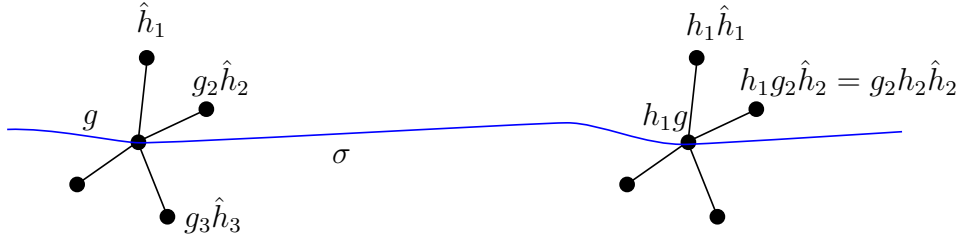
  \caption{Showing that the points of $\gamma$ lie in $J$.}
  \label{fig:quasiaxis}
\end{figure}

Now, let $A$ be a horoball $(R-4\delta+3)$--penetrated by $\sigma$, and note that $R-(4\delta+3) > C + \lambda +2\delta+ 1$.  Any point on $\sigma$ lies within $2\delta$ of some point of geodesics $[g,\hat{h}_1], [\hat{h}_1,h\hat{h}_1], [h\hat{h}_1,h_1g]$.  However, the first and third of these geodesics are between points in the Cayley graph and have length at most $C$.  Therefore, 
any points of $\sigma$ at depth greater than $R-(4\delta+3)$ in $A$ must be within $2\delta$ of the geodesic between $\hat{h}_1$ and $h \hat{h}_1$.  This implies that there is a geodesic with endpoints in $H$ which $(\lambda+1)$--penetrates $A$, which by $\lambda$-quasiconvexity implies that $A$ contains points at depth greater than $0$ in the image $\check\iota(X_H)$.  Condition \eqref{RHull:horoball} from Definition \ref{def:RHull} (along with the requirement that an $R$--hull be a full subgraph) ensures that 
 the intersection of $\widetilde{Z}$ with $A$ consists of a collection of vertical lines together with any horizontal edges connecting them.  In particular, the (super-regular) subsegment of $\sigma$ meeting $A$ is contained in $\widetilde{Z}$.  An exactly analogous argument shows that this subsegment is contained in $g_i\widetilde{Z}$ for each $i$, so all of $\sigma$ is contained in $J$.

This implies that $\sigma$ projects to a loop in $S_n$ of the type desired; if $C$ is the component containing the image of $\sigma$, then $\tau_{C,i}(\pi_1(C,p))$ contains a conjugate of the loxodromic $h_1$ for each $i$.

The other direction, that relative height dominates relatively multiplicity, is almost exactly the same as in \cite[Appendix A]{VH}.  The only difference is that we assume that the intersection is infinite and non-parabolic, so that it contains a loxodromic element.  This loxodromic element is then the one required by Definition \ref{d:rel mul}.
\end{proof}

\begin{corollary} \cite[Theorem 1.4]{hruskawise:packing}
The relative height of a relatively quasiconvex subgroup of a relatively hyperbolic group is finite.
\end{corollary}
\begin{proof}
  If the relative multiplicity is $n$, then in particular, $S_n$ contains a loop with a vertex at depth $0$.  Since $S_n$ avoids the fat diagonal, this vertex represents an $n$--tuple of distinct depth $0$ vertices of $Z$.  There are only finitely many such vertices, so the relative height is bounded.
\end{proof}
\subsection{Non-increasing of height under wide fillings}

Suppose that $\widetilde{Z}$ is an $R$--hull for the action of $H$ on $X$.  The following is an analog of \cite[Lemma A.45]{VH}.

\begin{lemma}[cf. Lemma A.45, \cite{VH}] \label{l:embed}
For all sufficiently long and $H$--wide fillings $\phi \co G \to G(N_1,\ldots , N_m)$, if $K = \ker(\phi)$, $K_H = K \cap H$, and $k \in K \smallsetminus K_H$, then $k . \widetilde{Z} \cap \widetilde{Z} = \emptyset$.
\end{lemma}
\begin{proof}
  By conditions \eqref{RHull:horoball} and \eqref{RHull:WGH} of an $R$--hull, there is some $R'$ so that if $g\widetilde{Z}\cap \widetilde{Z}\neq\emptyset$, and if $N$ is the $R'$--neighborhood of $H$ in the Cayley graph of $G$, then $g N \cap N\neq\emptyset$.  It follows that the set of $g$ for which $g\widetilde{Z}\cap \widetilde{Z}\neq\emptyset$ is contained in a finite union of double cosets
\[ A = \bigsqcup_{i=0}^l H g_i H, \mbox{ with } g_0 = 1. \]
  Now let $\phi$ be long and $H$--wide enough to apply Proposition \ref{p:weak sep} to all the elements $g_1,\ldots,g_l$.  For such a filling we have $\phi(g_i)\not\in\phi(H)$ for each $i$.  Equivalently, there is no $k\in K$ of the form $g_i h$ for $h\in H$ and $i>0$.

  Suppose by way of contradiction that $k\in (K\setminus K_H)\cap A$.  Then we can write $k = h_1 g_i h_2$ for some $i>0$ and some $h_1,h_2\in H$.  Conjugating we obtain a $k' \in (K\setminus K_H)\cap A$ which lies in $g_i H$.  But this contradicts the last paragraph.
\end{proof}
\begin{remark}
  Lemma A.45 in \cite{VH} is a special case of Lemma \ref{l:embed}.  The proof given in \cite{VH} contains the erroneous assertion that $A$ is a finite union of \emph{left} cosets; otherwise the proof given there is similar to our proof here of Lemma \ref{l:embed}, but using a theorem about $H$--fillings \cite[A.43]{VH} in place of our Proposition \ref{p:weak sep}.  
\end{remark}

Suppose $\pi\co (G,\mc{P})\to (\overline{G},\overline{\mc{P}})$ is a Dehn filling, and $X(G,\mc{P})$ is the combinatorial cusped space for $(G,\mc{P})$.  If $K$ is the kernel of the quotient map $G\to \overline{G}$, then the quotient $\overline{X} = \leftQ{X(G,\mc{P})}{K}$ is very nearly equal to the cusped space for the pair $(\overline{G},\overline{\mc{P}})$, differing only in the addition of some self-loops.  In particular, their $0$--skeleta are isometric, and we can safely ignore the difference.

Putting Lemma \ref{l:embed} together with uniformity of hyperbolicity and quasiconvexity after long Dehn fillings, we can prove the following:
\begin{lemma}\label{lem:image is Rhull}
  Fix $(G,\mc{P})$ relatively hyperbolic, and a relatively quasiconvex subgroup $H$.  
  For all $R$, there is an $R'$ satisfying the following:
  For all sufficiently long and $H$--wide fillings $\phi\co G \to G(N_1,\ldots , N_m)$, if $K = \ker(\phi)$, if $\tilde{Z}$ is an $R'$--hull for $H$, then $\widetilde{\overline{Z}}\subset \leftQ{X}{K}$ is an $R$--hull for the image of $H$ in $G(N_1,\ldots,N_m)$.
\end{lemma}
\begin{proof}
Let $\delta$ be such that $\overline{X}=\leftQ{X}{K}$ is $\delta$--hyperbolic whenever $K$ is the kernel of a sufficiently long filling (see Proposition \ref{p:agm 2.3}).  As discussed above this quotient is essentially equal to the combinatorial cusped space for the pair $(\overline{G},\overline{\mc{P}})$, where $\overline{G} = G(N_1,\ldots,N_m)$, and $\overline{\mc{P}}$ consists of the images of the elements of $\mc{P}$.
  Let $\lambda'$ be the constant from Proposition \ref{p:pi(H) QC}, so that $\phi(H)$ is $\lambda'$--relatively quasiconvex for a sufficiently long and $H$--wide filling.  Let $R_0 = R_0(\lambda',\delta)$ be such that any bi-infinite geodesic with endpoints in the limit set of a $\lambda'$--quasiconvex subset of a $\delta$--hyperbolic space is contained in the $R_0$--neighborhood of that quasiconvex subset.
Let $C = \max\{d_X(1,c_D)\}$, where $c_D$ ranges over the elements chosen in Section \ref{s:background}.
  Finally we fix some $R' > 3R + 2R_0 + C$.

  We assume that $\phi$ is sufficiently long and $H$--wide so that the results from the last paragraph apply.  

  We suppose that $\tilde{Z}$ is an $R'$--hull, and show that the image $\widetilde{\overline{Z}}\subset \leftQ{X}{K}$ is an $R$--hull for the image $\overline{H}$ of $H$ in $G(N_1,\ldots,N_m)$.

  Conditions \eqref{RHull:basepoint} and \eqref{RHull:horoball} of Definition \ref{def:RHull} follow easily from the fact that $\tilde{Z}$ is an $R'$--hull.  Condition \eqref{RHull:WGH} is a fairly straightforward consequence of the fact that $\overline{H}$ is relatively quasiconvex.

  We now establish Condition \eqref{RHull:geodesic}.  Suppose that $\gamma$ is a bi-infinite geodesic with endpoints in $\Lambda(\overline{H})$.  Let $p\in N_R(\gamma)\cap N_R(\overline{G})$.  Since $\check{\iota}(X_{\overline{H}})$ is $\lambda'$--quasiconvex, we have $d_{\overline{X}}(p,x)\leq R_0+R$ for some $x\in \check{\iota}(X_{\overline{H}})$.  Since the depth of $p$ is at most $R$, the depth of $x$ is at most $R_0 + 2R$.  Thus there is some $\overline{h}\in \overline{H}$ with $d_{\overline{X}}(x,\overline{h})\leq R_0 + 2R + C$.  Choose $h\in H$ projecting to $\overline{h}$, and note that there is a bi-infinite geodesic passing through $h$ with endpoints in $\Lambda(H)$.

  (We remark that because we assumed that $\gamma$ exists, $\Lambda(\overline{H})$ contains more than one point.  It follows that $\Lambda(H)$ contains more than one point, which implies the existence of such a bi-infinite geodesic.)

  In particular, an $R'$--ball about $h$ is contained in the $R'$--hull $\tilde{Z}$.  Since $R'> d_{\overline{X}}(\overline{h},p)$, the image of $\tilde{Z}$ in $\leftQ{X}{K}$ must contain $p$.
\end{proof}

We now prove that the relative height of $H$ does not increase under sufficiently long and $H$--wide fillings.

\begin{theorem}(cf. \cite[A.46]{VH}) \label{t:rh decrease}
For sufficiently long and $H$--wide fillings, the relative height of $(\overline{H},\overline{\mc{D}})$ in $(\overline{G},\overline{\mc{P}})$ is at most the relative height of $(H,\mc{D})$ in $(G,\mc{P})$.
\end{theorem}

\begin{proof}
  As usual, let $\delta$ be a constant so that the cusped space of $(G,\mc{P})$ and also those of sufficiently long fillings, are $\delta$--hyperbolic, and let $\lambda$ be a quasi-convexity constant for $H$, which we also assume (using Proposition \ref{p:pi(H) QC}) is a quasi-convexity constant for the image of $H$ under sufficiently long and $H$--wide fillings.

Let $R$ be sufficiently large to apply Theorem \ref{t:rh is rm} with these values of $\delta$ and $\lambda$.  Let $R'$ the the constant (depending on $R$) from the conclusion of Lemma \ref{lem:image is Rhull}.
  
Consider the following commutative diagram, which is equivariant with respect to the group actions and the natural maps between the groups (inclusion and quotient maps, as appropriate): 

\cd{
\widetilde{Z} \ar@(ul,u)[]^H\ar[r] \ar[d] & X \ar@(u,ur)[]^G\ar[d] \\
\widetilde{\overline{Z}}\ar[r]\ar@(l,ul)[]^{H/K_H} & \overline{X}\ar@(ur,r)^{G/K},
}
where $X$ is the cusped space for $(G,\mc{P})$, $\overline{X} = \leftQ{X}{K}$, $\widetilde{Z}$ is an $R$--hull for $H$, $K_H = H \cap K$ is the kernel of the induced filling on $H$, and $\widetilde{\overline{Z}} = \leftQ{\widetilde{Z}}{K_H}$.  

It follows immediately from Lemma \ref{l:embed} that $\widetilde{\overline{Z}}$ embeds in $\overline{X}$, and it follows from Lemma \ref{lem:image is Rhull} that $\widetilde{\overline{Z}}$ is an $R$--hull for $H/K_H$ in $\overline{X}$.  

Taking quotients by the relevant groups we get the diagram,
\cdlabel{ZvbarZ}{
Z \ar[d] \ar[r]^i & Y \ar[d]\\
\overline{Z} \ar[r]^{\overline\imath} & \overline{Y}}
where the horizontal maps
are immersions inducing the inclusions $H\to G$ and $\overline{H}\to
\overline{G}$ on the level of fundamental group.  The vertical maps from $Y$ to $\overline{Y}$ and from $Z$ to $\overline{Z}$ are homeomorphisms.  Theorem \ref{t:rh is rm} implies that the relative multiplicities of $Z$ in $Y$ and of $\overline{Z}$ in $\overline{Y}$ measure the relative heights of $H$ in $G$ and of $H/K_H$ in $G/K$, respectively.

For $n > 0$, define
\[  S_n = \left\{(z_1,\ldots,z_n)\in Z^n\mid
i(z_1)=\cdots=i(z_n)\right\}\setminus\Delta \]
and
\[   \overline{S}_n = \left\{(z_1,\ldots,z_n)\in \overline{Z}^n\mid
\overline\imath(z_1)=\cdots=\overline\imath(z_n)\right\}\setminus\overline\Delta,
\]  
where $\Delta = \left\{ (z_1, \ldots , z_n) \mid \mbox{ there exist $i \ne j$ so that } z_i = z_j \right\}$ is the fat diagonal in $Z^n$ and $\overline\Delta$ is the fat diagonal in $\overline{Z}^n$.

For each $i \in \{ 1 ,\ldots , n\}$
and each component $C$ of $S_n$ the projections of $Z^n$
to its factors induce maps
\[	\tau_{C,i} \co \pi_1(C) \to H	,	\]
and
\[	\overline\tau_{C,i} \co \pi_1(C) \to \overline{H}	.	\]
Since the quotient $Z = \leftQ{\widetilde{Z}}{H}$ can also be thought of as
$\leftQ{\widetilde{\overline{Z}}}{(H/K_H)}$,
the homomorphisms
$\overline\tau_{C,i}$ all factor as $\overline\tau_{C,i} = \phi|_H \circ \tau_{C,i}$, where $\phi$ is the filling map.

In particular, if $\gamma$ is a loop
in $\overline{S}_n$ so that $\overline\tau_{C,i}\left( [\gamma ] \right)$ is infinite for each $i \in \{ 1 , \ldots , n \}$ 
then it must be that $\tau_{C,i} \left( [\gamma ] \right)$ is already infinite for each $i$. 
\end{proof}

In case parabolic subgroups of a relatively hyperbolic group are finite, the relative height is the same as the height.  Recall that a filling $G \to G(N_1,\ldots , N_m)$ is \emph{peripherally finite} if for each $1 \le j \le m$ the subgroup $N_j$ has finite-index in $P_j$.  The following is an immediate consequence of Theorem \ref{t:rh decrease}.

\begin{corollary} \label{c:pf rh vs h}
For sufficiently long and $H$--wide peripherally finite fillings, the height of $\overline{H}$ in $\overline{G}$ is at most the relative height of $(H,\mc{D})$ in $(G,\mc{P})$.
\end{corollary}

\subsection{A result required by Wilton and Zalesskii}

\begin{definition} \cite[Definition 4.1]{WiltonZalesskii2}
Suppose that $(G,\mc{P})$ is relatively hyperbolic and $H \le G$.  We say that $H$ is {\em relatively malnormal} if for any $g \not\in H$ the intersection $H^g \cap H$ is conjugate into some element of $\mc{P}$.
\end{definition}
If $(G,\mc{P})$ is relatively hyperbolic and $G$ is torsion-free, then a relatively malnormal subgroup is either parabolic or a subgroup of relative height $1$.

For his joint work with Zalesskii \cite{WiltonZalesskii2}, Henry Wilton asked us to prove the following result, which appeared as \cite[Theorem 4.4]{WiltonZalesskii2}:

\begin{theorem} \label{t:Henry}
Let $G$ be a toral relatively hyperbolic group with parabolic subgroups $\{ P_1, \ldots , P_n \}$ and let $H$ be a subgroup which is relatively quasi-convex and relatively malnormal.  There exist subgroups of finite index $K_i' \subset P_i$ (for all $i$) such that, for all subgroups of finite index $L_i \subset K_i'$, if
\[	\eta \co G \to Q = G / \llangle L_1 ,\ldots , L_n \rrangle	\]
is the quotient map, the quotient $Q$ is word-hyperbolic and the image $\eta(H)$ in $Q$ is quasi-convex and almost malnormal.
\end{theorem}
\begin{proof}
We restrict to peripherally finite fillings.  For sufficiently long peripherally finite fillings
$\eta \co G \to Q$, the quotient $Q$ is hyperbolic (since it is hyperbolic relative to finite groups) by \cite[Theorem 7.3.(2)]{rhds}.  Since the peripheral subgroups of $Q$ are finite, there is no difference between quasiconvex and relatively quasiconvex subgroups.

By Proposition \ref{p:pi(H) QC}, for sufficiently long and $H$--wide fillings $\eta \co G \to Q$ the image 
$\eta(H)$ is quasi-convex, and by Corollary \ref{c:pf rh vs h}, for sufficiently long and $H$--wide fillings $\eta(H)$ is almost malnormal in $Q$.

It remains only to note that sufficiently long and $H$--wide peripherally finite fillings exist by Lemma \ref{l:PF wide fillings}, since the peripheral subgroups of $G$ are free abelian, and hence ERF.
\end{proof}

\small
\bibliographystyle{abbrv}

\end{document}